\documentclass[11pt]{amsart}
\usepackage[a4paper,vmargin={3.5cm,3.5cm},hmargin={2.5cm,2.5cm}]{geometry}
\usepackage[T1]{fontenc}
\usepackage[utf8]{inputenc}
\usepackage{stmaryrd}
\usepackage{mathtools}
\usepackage{graphicx}
\usepackage{framed}
\usepackage[normalem]{ulem}
\usepackage{amsmath,amsthm}
\usepackage{physics}
\usepackage{tikz}
\usepackage{mathdots}
\usepackage{yhmath}
\usepackage{cancel}
\usepackage{color}
\usepackage{siunitx}
\usepackage{array}
\usepackage{multirow}
\usepackage{mathtools}
\usepackage{dsfont}
\usepackage{mathrsfs}

\usepackage{amsthm}
\usepackage{amssymb}
\usepackage{gensymb}
\usepackage{tabularx}
\usepackage{extarrows}
\usepackage{booktabs}
\usetikzlibrary{fadings}
\usetikzlibrary{patterns}
\usetikzlibrary{shadows.blur}
\usetikzlibrary{shapes}
\usepackage{bbm}
\usepackage{mathtools}
\usepackage{forest}
\usepackage{thmtools}
\usepackage{thm-restate}
\usepackage[shortlabels]{enumitem}
\usepackage{hyperref}
\usepackage{cleveref}

\declaretheorem[name=Theorem]{theorem}
\declaretheorem[name=Proposition,numberwithin=section]{prop}
\declaretheorem[name=Corollary]{coro}

\newtheorem{lemma}[prop]{Lemma}
\newtheorem{remark}[prop]{Remark}

\newtheorem{definition}[prop]{Definition}

\newtheorem{propA}{Proposition}[section]  
\newtheorem{remarkA}{Remark}[section]  
\newtheorem{claimA}{Claim}[section]  

\forestset{multiple directions/.style={for tree={#1}, phantom, for relative level=1{no edge, delay={!c.content/.pgfmath=content("!u")}, before computing xy={l=0,s=0}}},
    multiple directions/.default={},
    grow subtree/.style={for tree={grow=#1}}, 
    grow' subtree/.style={for tree={grow'=#1}}}
\usepackage{tikz}
\usepackage{capt-of}
\usepackage{tikz-cd}
\newcommand{\eps}{\varepsilon}
\usepackage[maxbibnames=99]{biblatex}
\bibliography{References.bib}

\makeatletter
\def\paragraph{\@startsection{paragraph}{4}%
  \z@\z@{-\fontdimen2\font}%
  {\normalfont\bfseries}}
\makeatother

\begin{document}
\title[Uniqueness and locality]{Uniqueness and locality of the ground state of the disordered Monomer-Dimer models on independently weighted Unimodular Bienaymé-Galton-Watson trees}
\author{Mihyun Kang and Mike Liu
}

\address{Graz University of Technology, Institute of Discrete Mathematics, Steyrergasse 30, 8010 Graz, Austria}

\email{kang@math.tugraz.at, mike.liu@tugraz.at}

\begin{abstract}
    
    Consider a finite graph $G=(V(G),E(G))$ and two continuous weight distributions $\omega$ and $\xi$, for which we assume that $\xi$ is lower bounded. Next, independently draw weights $(w(e))_{e \in E(G)}$ with distribution $\omega$ on edges and $(x(v))_{v \in V(G)}$ with distribution $\xi$ on vertices. The ground state of the monomer-dimer model on the weighted graph $G$ is a collection of edges (dimers) and vertices (monomers) such that every vertex is included in at most one monomer or dimer, and such that the sum of weights on its dimers and monomers is maximised. 
    Take $(G_n,o_n)_{n \in \mathbb{N}}$ to be a sequence of random rooted weighted graphs that converges locally to an independently weighted unimodular Bienaymé-Galton-Watson tree $(\mathbb{T},o)$ with edge-weight distribution $\omega$ and vertex-weight distribution $\xi$. By proving that the ground state of the monomer-dimer model on the tree $(\mathbb{T},o)$ is almost surely unique and locally approximable, we prove that the ground state of the monomer-dimer model on $(G_n,o_n)$ must converge locally to the ground state of the monomer-dimer model on $(\mathbb{T},o)$. This also implies a strong decorrelation property of monomer-dimer models on unimodular Bienaymé-Galton-Watson trees.
\end{abstract}

\maketitle

\section{Introduction}

In the field of study of disordered systems, physicists have made many remarkable conjectures with the replica method. In the case of \textit{replica symmetric} models, the justification of such methods relies on the locality of an underlying message-passing cavity equation (see the monograph \cite{Beliefphysics} by Marc Mézard and Andrea Montanari). For the monomer-dimer model on Bienaymé-Galton-Watson trees with constant weights at positive monomer activity, this behaviour was proven to be correct by Alberici and Contucci \cite{AlbericiContucci}. Furthermore, for the maximum weight matching problem on the tree, which corresponds to the limit as the monomer activity goes to zero, a criterion on the degree distribution of the tree for decorrelation to occur was established by Bordenave, Lelarge and Salez \cite{bordenave2012matchings}.
On the side of the dense regime, which corresponds to the random assignment problem, several conjectures formalised by Marc Mézard and Giorgio Parisi \cite{Mparisi1,Mparisi2} were later proven correct by David Aldous \cite{Aldous1992AsymptoticsIT,aldous2000zeta2}. The fact that the underlying belief propagation procedure admits a unique measurable solution was later established by Antar Bandyopadhyay \cite{logisticendogeny}. The soundness of the replica method for the random assignment problem for more general weight distributions
was then proven by Johan Wästlund \cite{WastlundReplica}, and the uniqueness in law of the belief propagation procedure was obtained by Justin Salez \cite{SalezCavitypseudo}.

In the case of i.i.d. weighted sparse graphs, little is proven about the replica symmetric behaviour of maximum matchings. David Gamarnik, Tomasz Nowicki and Grzegorz Swirszcz \cite{gamarnik2003maximum} proved that correlation decay holds for maximum weight matchings on Erd\H{o}s-R\'enyi and $d-$regular random graphs when the weights are exponentially distributed and gave a general criterion that links correlation decay to the uniqueness of solutions to an iterated operator corresponding to the stationary measure of the cavity equation.
A recent paper by Arnab Sen and Wai-Kit Lam \cite{correlationarnab} shows that correlation decay for the belief propagation procedure holds for graphs of degree bounded by three or trees of bounded degree when the weights are i.i.d. and exponentially distributed.

A main difficulty of establishing correlation decay in maximum matchings on sparse graphs comes from the fact that, as opposed to the dense regime, the space of weight distributions that one can allow is essentially every non-atomic integrable distribution. Analytical solutions are thus very difficult to compute outside of specific cases, such as the aforementioned exponential weight distribution. As far as we know, exponential weight distributions are the only case where one obtains explicit solutions to the cavity equation in the sparse regimes.
The space of random sparse trees is also much larger than the previously considered class of trees represented by the Poisson Weighted Infinite Trees with intensity of the form $t \mapsto t^d\mathbbm{1}_{t \geq 0}$ for $d>1$, whose underlying unweighted graph is essentially deterministic. Even when restricted to unimodular Bienaymé-Galton-Watson trees, this allows for an additional parameter for the degree distribution, which is infinite-dimensional.
As a result, correlation decay in sparse graphs has been left as an open problem (see \cite[Question 1.2]{correlationarnab}).

In a recent paper by Nathanaël Enriquez, Laurent Ménard, Mike Liu, and Vianney Perchet \cite{enriquez2024optimalunimodularmatching}, the local weak convergence of maximum matching problems on random graphs converging locally weakly to an i.i.d. weighted unimodular Bienaymé-Galton-Watson tree with integrable weight distribution was proven. Their proof was based on a generalisation of Aldous' original method in \cite{Aldous1992AsymptoticsIT,aldous2000zeta2} and goes to great lengths in order to circumvent the inability to establish correlation decay of the maximum matching problem. The uniqueness in law of solutions of the cavity equation was established without answering the question of the almost sure uniqueness of the cavity messages on the tree itself.

In this paper, we achieve two different goals. We first generalise the study of maximum weight matching problems in sparse random graphs to the ground state of monomer-dimer models on i.i.d. weighted Unimodular-Bienaymé-Galton-Watson trees by allowing vertices to have any lower-bounded weights and attributing to any matching the sum of weights of its unmatched vertices. This recovers maximum weight matchings by simply setting vertex weights to be identically $0$. Second, we give a short and concise proof that correlation decay occurs for the ground state of monomer-dimer models. This proof is based on constructing an analytical invariant over extremal families of messages that appeared in the criterion given by David Gamarnik, Tomasz Nowicki and Grzegorz Swirszcz in \cite{gamarnik2003maximum}. Let us  note that the study of correlation decay has been carried out for ground states of monomer-dimer models on i.i.d. weighted $\mathbb{Z}^d$ for a broad class of weight distributions by Krishnan Kesav and Ray Gourab \cite{KesavGourabCLTZd}. However, their setting is quite different from ours, as the unweighted graph studied in \cite{KesavGourabCLTZd} is not random and not a tree. Furthermore, when translated into the maximum weight monomer-dimer configuration problem, they require the edge-weight distribution to admit no lower bound. The proof there also leans into heavy unimodular techniques and does not rely on message-passing arguments. Similar results about the positive temperature regime were also obtained (see \cite{DisorderedMDCylinder,ArnabDisorderedMD}).

As a consequence, we strengthen the results of \cite{enriquez2024optimalunimodularmatching} and answer the question on correlation decay \cite[Question 1.2]{correlationarnab} in a positive manner in the case of unimodular Bienaymé-Galton-Watson trees. Instead of weak convergence obtained in \cite{enriquez2024optimalunimodularmatching}, we show that if a sequence of weighted random rooted graphs converges locally strongly to an i.i.d. weighted unimodular Bienaymé-Galton-Watson tree, then the ground states of the monomer-dimer model on the weighted graphs must also converge locally strongly to the ground state of the monomer-dimer on the weighted tree that we identify thanks to the almost sure unique solution to the cavity equation on the weighted tree. Furthermore, we remove the integrability condition on the weights, which was required in \cite{enriquez2024optimalunimodularmatching}. We now state our first theorem below. 

\begin{theorem}\label{th:convergenceMD}
    Let $(\mathbb{T},w,x,o)$ be a unimodular Bienaymé-Galton-Watson tree with reproduction law $\pi$ with edge weights $(w(e))_{e \in E(\mathbb{T})}$ with distribution $\omega$ and vertex weights $(x(v))_{v \in V(\mathbb{T})}$ with distribution $\xi$.
    Assume that:
    \begin{enumerate}[(i)]
    \item The vertex-weight distribution $\xi$ is lower bounded. In addition, it is either continuous or admits at most one atom at its minimum $m_{\xi}\in \mathbb{R}$.
    \item The edge-weight distribution $\omega$ is continuous. In addition, its support is locally convex under its maximum $M_{\omega} \in \mathbb{R} \cup \{+\infty\}$.
    \item The degree distribution $\pi$ admits a second moment.
    \end{enumerate}
    Then there almost surely exists a unique optimal monomer-dimer configuration $\mathbb{MD}$ on $(\mathbb{T},w,x,o)$ in the following sense:
    Assume that $(G_n,w_n,x_n,o_n)_{n \in \mathbb{N}}$ is a sequence of weighted random rooted graphs converging locally to $(\mathbb{T},w,x,o)$. Denote by $\mathrm{MD}_n$ any optimal monomer-dimer configuration on $G_n$. Then $(G_n,w_n,x_n,\mathrm{MD}_n,o_n)$ converges locally to $(\mathbb{T},w,x,\mathbb{MD},o)$.
\end{theorem}

Using Skorokhod's representation theorem, we get the equivalent proposition for local weak convergence.

\begin{coro}\label{coro:weakconvergenceMD}
   Let $(\mathbb{T},w,x, \mathbb{MD},o)$ be as in Theorem~\ref{th:convergenceMD}.
    Let $(G_n,o_n)_{n \in \mathbb{N}}$ be any sequence of rooted random graphs that converge locally weakly to $(\mathbb{T},o)$. Draw i.i.d. edge and vertex weights $w_n$ and $x_n$ on $G_n$ according to distribution $\omega$ and $\xi$. Denote by $\mathrm{MD}_n$ any optimal monomer-dimer configuration on $G_n$. Then $(G_n,\mathrm{MD}_n,o_n)$ converges locally weakly to $(\mathbb{T},\mathbb{MD},o)$. 
\end{coro}

When $\xi=\delta_0$, Corollary~\ref{coro:weakconvergenceMD} is one of the main findings of \cite{enriquez2024optimalunimodularmatching}.
We also establish correlation decay for the ground state of monomer-dimer models themselves, which we formalise in the next theorem.

\begin{theorem}\label{th:decay}
    Let $(\mathbb{T},w,x, \mathbb{MD},o)$ be as in Theorem~\ref{th:convergenceMD}.
    Let $(G_n,w_n,x_n)_{n \in \mathbb{N}}$ be a sequence of weighted random graphs and $\mathrm{MD}_n$ be any optimal monomer-dimer configuration on $(G_n,w_n,x_n)$.

    Take $o_n,o_n'$ two vertices in $G_n$. Assume that:
    \begin{enumerate}
        \item The graph distance between $o_n$ and $o_n'$ goes to infinity as $n \to \infty$.
        \item The pair \[\bigg((G_n,w_n,x_n,o_n),(G_n,w_n,x_n,o_n')\bigg)\] jointly locally converges to two independent copies of $(\mathbb{T},w,x,o)$.
    \end{enumerate}
    Then the pair
    \[ \bigg( (G_n,w_n,x_n,\mathrm{MD}_n,o_n),(G_n,w_n,x_n,\mathrm{MD}_n,o_n') \bigg)  \]
    jointly locally converges to two independent copies of $(\mathbb{T},w,x,\mathbb{MD},o)$.
\end{theorem}

\begin{remark}
    The additional conditions (1)--(2) in Theorem~\ref{th:decay} are necessary to avoid pathological cases such as $G_n$ being $n$ copies of the same subcritical Bienaymé-Galton-Watson tree.
\end{remark}

Since the correlation between states of edges and vertices in the ground state of the monomer-dimer model decays, this also allows us to deduce laws of large numbers on these ground states as a corollary. 
\begin{coro}\label{coro:averageweight}
    Let us set ourselves in the context of local weak convergence of Corollary~\ref{coro:weakconvergenceMD}.
    Assume that the conditions of Theorem~\ref{th:decay} hold for two uniformly chosen vertices $o_n$ and $o_n'$ in $G_n$, with (2) replaced with local weak convergence. Let $\mathrm{MD}_n$ be any optimal monomer-dimer configuration on $G_n$. Let $f$ be any bounded local function on rooted marked graphs. Then
    \begin{align*}
        \frac{1}{|V(G_n)|}\left( \sum_{v \in V(G_n)} f(G_n,w_n,x_n,\mathrm{MD}_n,v)  \right) \quad \underset{ n \rightarrow +\infty}{\overset{\mathbb{P}}{\longrightarrow}} \quad \mathbb{E}\left[ f(\mathbb{T},w,x,\mathbb 
        {MD},o)\right].
    \end{align*}
\end{coro}
\begin{remark}
    A law of large numbers can then be deduced for the total weight. Indeed, fix some constant $K>0$ and set the localising function $\tau_K: t \mapsto \max(-K,\min(K,t))$. Apply Corollary~\ref{coro:averageweight} to the function
    \[f_K:(G,w,x,\mathrm{MD},o) \mapsto \tau_K\left(x(o)\mathbbm{1}_{o \text{ is a monomer in } \mathrm{MD}} + \sum_{v \sim o}\frac{1}{2}w(o,v)\mathbbm{1}_{ \{o,v\} \text{ is a dimer in }  \mathrm{MD}}\right),  \]    
    using monotone convergence theorem and uniform integrability to swap the order of integration, this becomes a law of large numbers on the average weight of the optimal monomer-dimer configuration on $G_n$. As $n$ goes to infinity, the renormalised total weight given by
    \begin{align*}
        \frac{1}{|V(G_n)|}\left(\sum_{v \text{ is a monomer in } \mathrm{MD}_n}x_n(v)+ \sum_{e \text{ is a dimer in } \mathrm{MD}_n} w_n(e)  \right)
    \end{align*}
    converges in probability to
    \begin{align*}
        \mathbb{E}\left[ x(o)\mathbbm{1}_{o \text{ is a monomer in }\mathbb{MD}}+ \frac{1}{2}\sum_{v \sim o} w(o,v)\mathbbm{1}_{\{o,v\}\text{ is a dimer in } \mathbb{MD}} \right].
    \end{align*}
    \end{remark}
Since we do not assume integrability of the weights, the right hand side can be infinite and the convergence still holds, where convergence in probability to $+\infty$ means that for any constants $C>0,\eps>0$, there exists $n$ large enough from which the probability that the random variable on the left is bigger than $C$ is at least $1-\eps$. 

In Section~\ref{sec:prelim}, we specify the notations and the precise definitions we use in this paper. The objective of Section~\ref{sec:messagepassing} is to exhibit the message-passing procedure that builds the unique ground state of the monomer-dimer model on generically weighted finite trees.  Section~\ref{sec:subcritical} gives the proof, whose main idea is derived from \cite{gamarnik2003maximum}, of our main theorems conditionally on the validity of Proposition~\ref{prop:mainprop}, which reflects the uniqueness in distribution of the iterated cavity equation. The heart and main contribution of the paper then lies in Section~\ref{sec:mainprop}, where we provide a short proof of this uniqueness.

\bigskip

\paragraph{Acknowledgements}
Mike Liu would like to thank Arnab Sen for helpful discussions around correlation decay, and also Nathanaël Enriquez and Laurent Ménard for going over the first draft of Section~\ref{sec:mainprop}. This research was funded by the Austrian Science Fund (FWF) [10.55776/F1002, 10.55776/I6502]. For the purpose of open access, the author has applied a CC BY public copyright licence to any Author Accepted Manuscript version arising from this submission.

\section{Preliminaries}\label{sec:prelim}
We let $\mathbb{N}_0:=\mathbb N\cup \{0\}$. For any $\ell \in \mathbb N$ let $\mathbb{N}_{\geq \ell}:=\{\ell,\ell +1,\ldots\}$ and $[\ell]:=\{1,2,\ldots,\ell\}$.     

Let $G$ be a graph with vertex set $V(G)$ and edge set $E(G) \subseteq \binom{V(G)}{2}$. For $e=\{u,v\}\in E(G)$, we say that $u$ and $v$ are adjacent, denoted by $u\sim v$, and that $u$ and $v$ are incident to $e$. We often write $v\in G$ and $\{u,v\}\in G$ instead of $v\in V(G)$ and $\{u,v\}\in E(G)$. For a subgraph $G'$ of $G$, we denote by $G\setminus G'$ the graph with vertex set $V(G)\setminus V(G')$ and edge set $E(G)\setminus E(G')$. For $U \subseteq V(G)$, we let $G\setminus U$ denote the graph with vertex set $V(G)\setminus U$ and edge set $E(G) \cap \binom{V(G)\setminus U}{2}$. 

A (possibly infinite) graph $G$ is {\em locally finite} if every vertex of $G$ has finite degree. A {\em rooted graph} is a pair $(G,o)$, where $G$ is a locally finite graph and $o$ is a vertex of $G$. Two rooted graphs $(G,o)$ and $(G',o')$ are {\em isomorphic}, which is denoted by $(G,o) \cong (G',o')$, if there exists a bijection $\iota:V(G)\to V(G')$ such that $\{u,v\}\in E(G)$ iff  $\{\iota(u),\iota(v)\}\in E(G')$  and $\iota(o)=o'$. Given $H \in \mathbb{N}_0$ and a rooted graph $(G,o)$, the ball $B_{H}(G,o)$ is the subgraph of $G$ induced on the vertices with distance at most $H+1$ from $o$\footnote{Such a subgraph is often denoted by $B_{H+1}(G,o)$ in graph theory, but by $B_{H}(G,o)$ in this paper, because $H$ corresponds to the depth or the length of paths, used to update the message on an edge $\{o,v\}\in G$.}. The ball $B_{H}(G,o)$ is considered as a rooted graph with root $o$. We also call $B_{H}(G,o)$ the $H-$neighbourhood of $(G,o)$ and call $\partial B_H(G,o):=B_{H}(G,o)\setminus B_{H-1}(G,o)$ the $H-$boundary of $(G,o)$ or simply the boundary if it is clear from context. 

For any rooted graph $G$, we can define a set of directed edges $\overset{\rightarrow}{E}(G)$ by duplicating each edge and giving opposite orientations. Formally, $\overset{\rightarrow}{E}(G) := \bigcup_{\{u,v\} \in E(G)} \{ (u,v),(v,u) \}$. 

In a rooted tree $(T,o)$, we will say that a directed edge $\overset{\rightarrow}{e} \in \overset{\rightarrow}{E}(G)$ is pointing {\em outwards} if there exists a directed path $(\overset{\rightarrow}{e}_0,\ldots,\overset{\rightarrow}{e}_k)$ such that the source of $\overset{\rightarrow}{e}_0$ is $o$, and $\overset{\rightarrow}{e}_k=\overset{\rightarrow}{e}$. Inversely, we will say that the directed edge $\overset{\rightarrow}{e}$ is pointing {\em inwards} if there exists a directed path $(\overset{\rightarrow}{e}_0,\ldots,\overset{\rightarrow}{e}_k)$ such that $\overset{\rightarrow}{e}_0=\overset{\rightarrow}{e}$ and the endpoint of $\overset{\rightarrow}{e}_k$ is $o$.

\smallskip
\begin{definition}
    For a graph $G$, we can additionally equip edges and vertices with weights by considering two functions $w: E(G) \to \mathbb{R}$ and $x: V(G) \to \mathbb{R}$. We will then say that $(G,w,x)$ is a weighted graph.
    In order to alleviate notations, we will often write $w(u,v)$ or $w(v,u)$ to be the same value $w(\{u,v\})$.
\end{definition}

Two weighted rooted graphs $(G,w,x,o)$ and $(G',w',x',o')$ are {\em isomorphic}, which is denoted by $(G,w,x,o) \cong (G',w',x',o')$, if there exists an isomorphism $\iota:V(G)\to V(G')$ that additionally satisfies $w(\{u,v\})=w'(\{\iota(u),\iota(v)\})$ for every $\{u,v\}\in E(G)$ and $x(v)=x'(\iota(v))$ for every $v\in V(G)$.

On a rooted graph, we can define a monomer-dimer configuration as follows.
\begin{definition}
    For any graph $G$, we say that a subset of edges $\mathrm{MD} \subseteq E(G)$ is a {\em monomer-dimer configuration} on $G$, if for any vertex $v \in G$, there exists at most one vertex $u \in G$ adjacent to $v$ such that $\{u,v\} \in \mathrm{MD}$. We then call every element of  $\mathrm{MD}$ a {\em dimer} in $\mathrm{MD}$, and every vertex $v \in G$ such that there exists no element of $\mathrm{MD}$ incident to $v$ a {\em monomer} in $\mathrm{MD}$. 
\end{definition}

We can now define optimality of monomer-dimer configurations on a finite weighted graph as a maximiser of the sum of weights on monomers and dimers

\begin{definition}
    A monomer-dimer configuration $\mathrm{MD}$ on a finite weighted graph $(G,w,x)$ is said to be {optimal} if it maximises
    \[  \sum_{ e \text{ is a dimer in } \mathrm{MD}} w(e)+ \sum_{v \text{ is a monomer in }\mathrm{MD}} x(v).  \]
    We write $\mathrm{MD}_{\mathrm{opt}}(G,w,x)$ to be any arbitrarily chosen optimal monomer-dimer configuration on $(G,w,x)$. 
\end{definition}

\begin{remark}
The optimal monomer-dimer configuration also corresponds to the ground states of the weighted monomer-dimer model on $G$.
\end{remark}
Next, we provide a definition of local convergence that is standard in the literature (see \cite{vanderHofstadBook2}).
\begin{definition}\label{def:locconv}
    We say that a sequence of rooted graphs $(G_n,o_n)$ converges locally to another rooted graph $(G,o)$ if for any $H\in \mathbb{N}_0$, there exists $N\in \mathbb{N}$ large enough so that for any $n \geq N$, we have $B_H(G_n,o_n) \cong B_H(G,o)$.
    This convergence induces a metrisable topology called {\em local topology} on the space of rooted graphs, and a corresponding distance is
    \[ d_{\mathrm{loc}}\left( (G,o),(G',o')\right)= \frac{1}{1+R}\]
    where $R:=\inf \{ H\in \mathbb{N}_0: B_H(G,o) \cong B_H(G',o')  \} $.
    
    For weighted graphs, we will mimic the definition by saying that $(G_n,w_n,x_n,o_n)$ converges locally to $(G,w,x,o)$ if for any $H\in \mathbb{N}_0$, there exists $N\in \mathbb{N}$ large enough so that for any $n \geq N$, we have
    $B_H(G_n,w_n,x_n,o_n) \cong B_H(G,w,x,o)$.
    
    The weak convergence resulting from this topology on random rooted graphs is then called {\em weak local convergence}.
\end{definition}
\begin{remark}
    Since weights are continuous, it would be more natural to define local convergence by allowing some leeway on weights. For instance, we could instead state that for every $\eps>0$ and $H\in \mathbb{N}_0$, there exists $N\in \mathbb{N}$ large enough so that for any $n \geq N$, we both have that
    $B_H(G_n,o_n) \cong B_H(G,o)$ and that the sum of the difference of weights through the corresponding graph isomorphism is bounded above by $\eps$. This results in a topology that we call $L^1-$local topology of weighted graphs.  We voluntarily choose a more restrictive definition to avoid technicalities, as converting the proof with this less restrictive distance does not change the main idea of our proof.
\end{remark}

Finally, we introduce the definition of a local function which only depends on a neighbourhood of the root.

\begin{definition}\label{def:locfunc}
    Let $\mathcal{G}^{\star}$ be the space of rooted weighted graphs. We say that a function 
    \[  f:\mathcal{G}^{\star} \to \mathbb{R}  \]
    is a local function if there exists $H\in \mathbb{N}_0$ such that $B_H(G,w,x,o) \cong B_H(G',w',x',o')$ implies $f(G,w,x,o)=f(G',w',x',o')$.
\end{definition}

Definition~\ref{def:locconv} and Definition~\ref{def:locfunc} can be extended to a rooted weighted graph equipped with a monomer-dimer configuration $(G,w,x,\mathrm{MD},o)$ by identifying a monomer-dimer configuration with its membership function $e \in E(G) \mapsto \mathbbm{1}_{e \text{ is a dimer in }\mathrm{MD}}$. 
\subsection{Unimodular Bienaymé-Galton-Watson trees}
In this subsection, we introduce {\em Unimodular Bienaymé-Galton-Watson trees (UBGW)} along with models of random graphs that converge locally in law to these trees. Weights on edges and vertices will always be drawn independently according to distributions $\omega$ and $\xi$. These results being common folklore in the literature, we draw the presentation from \cite{enriquez2024optimalunimodularmatching}.

Let $\pi$ be a probability measure on $\mathbb{N}_0$ with finite expectation $E > 0$. Let $\hat{\pi}$ be the size-biased version of $\pi$, that is, $\forall k \in \mathbb{N}_0$, $\hat{\pi} (k) =\frac{k}{E} \pi (k)$. 

The vertex-rooted UBGW tree is the random tree $(\mathbb T,o)$ with the following law:
    \begin{itemize}
        \item The number of children of vertices in $\mathbb T$ are all independent.
        \item The number of children of the root $o$ is distributed according to $\pi$.
        \item Every non-root vertex has a number of children distributed according to $\hat \pi$.
    \end{itemize}
The random tree $(\mathbb T, o)$ is a vertex-rooted unimodular random graph. See Figure~\ref{fig:VUBGW} for an illustration.
\tikzset{every picture/.style={line width=0.75pt}} 
\begin{figure}[!ht]
\centering
\begin{tikzpicture}[x=0.75pt,y=0.75pt,yscale=-.8,xscale=.9]
\draw    (299.1,0.3) -- (130,110) ;
\draw    (299.1,0.3) -- (250,110) ;
\draw    (299.1,0.3) -- (350,111.5) ;
\draw    (299.1,0.3) -- (450,110) ;
\draw    (130,110) -- (100,192) ;
\draw    (130,110) -- (160,193) ;
\draw    (350,111.5) -- (290,190) ;
\draw    (350,111.5) -- (350,191) ;
\draw    (350,111.5) -- (410,192) ;
\draw    (450,190) -- (450,110) ;
\draw    (100,192) -- (80,260) ;
\draw    (100,192) -- (120,260) ;
\draw    (290,190) -- (290,260) ;
\draw    (410,192) -- (380,260) ;
\draw    (410,192) -- (410,260) ;
\draw    (410,192) -- (440,260) ;
\draw  [fill={rgb, 255:red, 7; green, 0; blue, 0 }  ,fill opacity=1 ] (124.77,110) .. controls (124.77,107.02) and (127.11,104.6) .. (130,104.6) .. controls (132.89,104.6) and (135.23,107.02) .. (135.23,110) .. controls (135.23,112.98) and (132.89,115.4) .. (130,115.4) .. controls (127.11,115.4) and (124.77,112.98) .. (124.77,110) -- cycle ;
\draw  [fill={rgb, 255:red, 7; green, 0; blue, 0 }  ,fill opacity=1 ] (293.88,0.3) .. controls (293.88,-2.68) and (296.21,-5.1) .. (299.1,-5.1) .. controls (301.99,-5.1) and (304.33,-2.68) .. (304.33,0.3) .. controls (304.33,3.28) and (301.99,5.7) .. (299.1,5.7) .. controls (296.21,5.7) and (293.88,3.28) .. (293.88,0.3) -- cycle ;
\draw  [fill={rgb, 255:red, 7; green, 0; blue, 0 }  ,fill opacity=1 ] (244.77,110) .. controls (244.77,107.02) and (247.11,104.6) .. (250,104.6) .. controls (252.89,104.6) and (255.23,107.02) .. (255.23,110) .. controls (255.23,112.98) and (252.89,115.4) .. (250,115.4) .. controls (247.11,115.4) and (244.77,112.98) .. (244.77,110) -- cycle ;
\draw  [fill={rgb, 255:red, 7; green, 0; blue, 0 }  ,fill opacity=1 ] (344.77,111.5) .. controls (344.77,108.52) and (347.11,106.1) .. (350,106.1) .. controls (352.89,106.1) and (355.23,108.52) .. (355.23,111.5) .. controls (355.23,114.48) and (352.89,116.9) .. (350,116.9) .. controls (347.11,116.9) and (344.77,114.48) .. (344.77,111.5) -- cycle ;
\draw  [fill={rgb, 255:red, 7; green, 0; blue, 0 }  ,fill opacity=1 ] (444.77,110) .. controls (444.77,107.02) and (447.11,104.6) .. (450,104.6) .. controls (452.89,104.6) and (455.23,107.02) .. (455.23,110) .. controls (455.23,112.98) and (452.89,115.4) .. (450,115.4) .. controls (447.11,115.4) and (444.77,112.98) .. (444.77,110) -- cycle ;
\draw  [fill={rgb, 255:red, 7; green, 0; blue, 0 }  ,fill opacity=1 ] (94.77,192) .. controls (94.77,189.02) and (97.11,186.6) .. (100,186.6) .. controls (102.89,186.6) and (105.23,189.02) .. (105.23,192) .. controls (105.23,194.98) and (102.89,197.4) .. (100,197.4) .. controls (97.11,197.4) and (94.77,194.98) .. (94.77,192) -- cycle ;
\draw  [fill={rgb, 255:red, 7; green, 0; blue, 0 }  ,fill opacity=1 ] (74.77,260) .. controls (74.77,257.02) and (77.11,254.6) .. (80,254.6) .. controls (82.89,254.6) and (85.23,257.02) .. (85.23,260) .. controls (85.23,262.98) and (82.89,265.4) .. (80,265.4) .. controls (77.11,265.4) and (74.77,262.98) .. (74.77,260) -- cycle ;
\draw  [fill={rgb, 255:red, 7; green, 0; blue, 0 }  ,fill opacity=1 ] (114.77,260) .. controls (114.77,257.02) and (117.11,254.6) .. (120,254.6) .. controls (122.89,254.6) and (125.23,257.02) .. (125.23,260) .. controls (125.23,262.98) and (122.89,265.4) .. (120,265.4) .. controls (117.11,265.4) and (114.77,262.98) .. (114.77,260) -- cycle ;
\draw  [fill={rgb, 255:red, 7; green, 0; blue, 0 }  ,fill opacity=1 ] (154.77,193) .. controls (154.77,190.02) and (157.11,187.6) .. (160,187.6) .. controls (162.89,187.6) and (165.23,190.02) .. (165.23,193) .. controls (165.23,195.98) and (162.89,198.4) .. (160,198.4) .. controls (157.11,198.4) and (154.77,195.98) .. (154.77,193) -- cycle ;
\draw  [fill={rgb, 255:red, 7; green, 0; blue, 0 }  ,fill opacity=1 ] (344.77,191) .. controls (344.77,188.02) and (347.11,185.6) .. (350,185.6) .. controls (352.89,185.6) and (355.23,188.02) .. (355.23,191) .. controls (355.23,193.98) and (352.89,196.4) .. (350,196.4) .. controls (347.11,196.4) and (344.77,193.98) .. (344.77,191) -- cycle ;
\draw  [fill={rgb, 255:red, 7; green, 0; blue, 0 }  ,fill opacity=1 ] (284.77,190) .. controls (284.77,187.02) and (287.11,184.6) .. (290,184.6) .. controls (292.89,184.6) and (295.23,187.02) .. (295.23,190) .. controls (295.23,192.98) and (292.89,195.4) .. (290,195.4) .. controls (287.11,195.4) and (284.77,192.98) .. (284.77,190) -- cycle ;
\draw  [fill={rgb, 255:red, 7; green, 0; blue, 0 }  ,fill opacity=1 ] (404.77,192) .. controls (404.77,189.02) and (407.11,186.6) .. (410,186.6) .. controls (412.89,186.6) and (415.23,189.02) .. (415.23,192) .. controls (415.23,194.98) and (412.89,197.4) .. (410,197.4) .. controls (407.11,197.4) and (404.77,194.98) .. (404.77,192) -- cycle ;
\draw  [fill={rgb, 255:red, 7; green, 0; blue, 0 }  ,fill opacity=1 ] (284.77,260) .. controls (284.77,257.02) and (287.11,254.6) .. (290,254.6) .. controls (292.89,254.6) and (295.23,257.02) .. (295.23,260) .. controls (295.23,262.98) and (292.89,265.4) .. (290,265.4) .. controls (287.11,265.4) and (284.77,262.98) .. (284.77,260) -- cycle ;
\draw  [fill={rgb, 255:red, 7; green, 0; blue, 0 }  ,fill opacity=1 ] (374.77,260) .. controls (374.77,257.02) and (377.11,254.6) .. (380,254.6) .. controls (382.89,254.6) and (385.23,257.02) .. (385.23,260) .. controls (385.23,262.98) and (382.89,265.4) .. (380,265.4) .. controls (377.11,265.4) and (374.77,262.98) .. (374.77,260) -- cycle ;
\draw  [fill={rgb, 255:red, 7; green, 0; blue, 0 }  ,fill opacity=1 ] (404.77,260) .. controls (404.77,257.02) and (407.11,254.6) .. (410,254.6) .. controls (412.89,254.6) and (415.23,257.02) .. (415.23,260) .. controls (415.23,262.98) and (412.89,265.4) .. (410,265.4) .. controls (407.11,265.4) and (404.77,262.98) .. (404.77,260) -- cycle ;
\draw  [fill={rgb, 255:red, 7; green, 0; blue, 0 }  ,fill opacity=1 ] (434.77,260) .. controls (434.77,257.02) and (437.11,254.6) .. (440,254.6) .. controls (442.89,254.6) and (445.23,257.02) .. (445.23,260) .. controls (445.23,262.98) and (442.89,265.4) .. (440,265.4) .. controls (437.11,265.4) and (434.77,262.98) .. (434.77,260) -- cycle ;
\draw  [fill={rgb, 255:red, 7; green, 0; blue, 0 }  ,fill opacity=1 ] (444.77,190) .. controls (444.77,187.02) and (447.11,184.6) .. (450,184.6) .. controls (452.89,184.6) and (455.23,187.02) .. (455.23,190) .. controls (455.23,192.98) and (452.89,195.4) .. (450,195.4) .. controls (447.11,195.4) and (444.77,192.98) .. (444.77,190) -- cycle ;
\draw (311,-5.3) node [anchor=north west][inner sep=0.75pt]    {$\pi $};
\draw (461,102.4) node [anchor=north west][inner sep=0.75pt]    {$\hat{\pi }$};
\draw (416,182.4) node [anchor=north west][inner sep=0.75pt]    {$\hat{\pi }$};
\draw (275,-5.6) node [anchor=north west][inner sep=0.75pt]    {$o$};
\draw (356,102.4) node [anchor=north west][inner sep=0.75pt]    {$\hat{\pi }$};
\draw (146,100.4) node [anchor=north west][inner sep=0.75pt]    {$\hat{\pi }$};
\draw (301,182.4) node [anchor=north west][inner sep=0.75pt]    {$\hat{\pi }$};
\draw (111,182.4) node [anchor=north west][inner sep=0.75pt]    {$\hat{\pi }$};
\draw (171,182.4) node [anchor=north west][inner sep=0.75pt]    {$\hat{\pi }$};
\draw (361,182.4) node [anchor=north west][inner sep=0.75pt]    {$\hat{\pi }$};
\draw (256,102.4) node [anchor=north west][inner sep=0.75pt]    {$\hat{\pi }$};
\draw (456,182.4) node [anchor=north west][inner sep=0.75pt]    {$\hat{\pi }$};
\draw (86,250.4) node [anchor=north west][inner sep=0.75pt]    {$\hat{\pi }$};
\draw (126,250.4) node [anchor=north west][inner sep=0.75pt]    {$\hat{\pi }$};
\draw (296,250.4) node [anchor=north west][inner sep=0.75pt]    {$\hat{\pi }$};
\draw (386,250.4) node [anchor=north west][inner sep=0.75pt]    {$\hat{\pi }$};
\draw (416,250.4) node [anchor=north west][inner sep=0.75pt]    {$\hat{\pi }$};
\draw (446,250.4) node [anchor=north west][inner sep=0.75pt]    {$\hat{\pi }$};
\end{tikzpicture}
\caption{A $2-$neighbourhood $B_2(\mathbb T,o)$ of a vertex-rooted UBGW tree $(\mathbb T,o)$ with the law $\pi$ of the number of children drawn on the root $o$ and the law $\hat{\pi}$ of the number of children on every non-root vertex.} \label{fig:VUBGW}
\end{figure}

\bigskip

The most classical examples of random graphs converging to UBGW trees we consider are sparse Erd\H{o}s–R\'enyi random graph and configuration models:
\begin{itemize}
    \item \textbf{Sparse Erd\H{o}s–R\'enyi random graph:} As introduced in the celebrated paper of Erdős and Rényi \cite{ErdosRenyi}, for $c>0$ and $N\in \mathbb{N}$, the random graph $\mathcal G (N,\frac{c}{N})$ is defined on the vertex set $[N]$ with independent edges between vertices with probability $\frac{c}{N}$. Once uniformly rooted, the sequence of these graphs $\left(\mathcal G (N,\frac{c}{N})\right)_{N\in \mathbb{N}}$ converges locally to a UBGW tree with reproduction law Poisson with parameter $c$, when $N$ goes to $\infty$.
    \item \textbf{Configuration model:} This model was introduced by Bollob\'as in 1980 \cite{B1980} and can be defined as follows.
Let $N\in \mathbb{N}$ and let $d_1, \ldots, d_N \in \mathbb{N}_0$ be such that $d_1 + \cdots +d_N$ is even. We interpret $d_i$ as the number of half-edges attached to vertex $i$. Then, the configuration model associated with the sequence $(d_i)_{1 \leq i \leq N}$ is the random multigraph with vertex set $[N]$ obtained by a uniform perfect matching of these half-edges. If $d_1 + \cdots + d_N$ is odd, we change $d_N$ into $d_N+1$ and do the same construction. Now, let $\mathbf{d}^{(N)}$ be a sequence of random variables defined on the same probability space $(\Omega,\mathcal{F},\mathbb{P})$ such that for every $N\in \mathbb{N}$, $\mathbf{d}^ {(N)} = (d_1^{(N)}, \ldots, d_N^{(N)}) \in \mathbb{N}_0^N$. Furthermore, suppose that there exists a probability measure $\pi$ on $\mathbb{N}_0$ with finite first moment such that
\[  \forall k \in \mathbb{N}_0, \quad \frac{1}{N} \sum\limits_{ j=1}^N \mathbbm{1}_{ d_j^{(N)} = k}  \underset{ \quad N \rightarrow +\infty \quad}{ \longrightarrow} \pi( \{k\}).   \]
The sequence of random configuration graphs associated with $\mathbf{d}^{(N)}$ has asymptotically a positive probability to be simple. In addition, this sequence of random graphs, when uniformly rooted, converges locally in law to the UBGW random tree with offspring distribution $\pi$, see \cite{vanderHofstadBook2} for more details.
\end{itemize}

\section{Message-passing for monomer-dimer ground states}\label{sec:messagepassing}
In this section, we present one of the main tools we use for studying optimal monomer-dimer configurations. In particular, there exists a message-passing algorithm that constructs the optimal solution in linear time on finite trees.

Let $(T,w,x)$ be a finite weighted tree, and assume that there is a unique optimal monomer-dimer configuration on $(T,w,x)$.
Let us consider an edge $e=\{u,v\}$. The principle is to look at the difference between the unique optimal monomer-dimer configuration of $(T,w,x)$ and the optimal monomer-dimer configuration under the constraint of $e$ being a dimer. Let us write $\mathrm{OPT}(T,w,x)$ for the weight of the optimal monomer-dimer configuration on the implicitly weighted $T$.

Let us write $T_{(v,u)}$ and $T_{(u,v)}$ to be the two connected components of $T\setminus {e}$ containing $u$ and $v$, respectively. The key is that the total weight of the optimal configuration \textbf{forcing} $e$ to be a dimer can be written as
\[ w(e)+ \mathrm{OPT}( T_{(v,u)}\setminus \{u\})+\mathrm{OPT}(T_{(u,v)} \setminus \{v\}) .\]
The total weight of the optimal configuration of $T\setminus e$ can simply be written as:
\[ \mathrm{OPT}(T_{(v,u)}) + \mathrm{OPT}(T_{(u,v)}).\]
From this, we deduce that $e$ is a dimer in the optimal configuration $\mathrm{MD}_{\mathrm{opt}}(T,w,x)$ if and only if the first quantity is bigger than the second one, in other words,
\[ w(e)+ \mathrm{OPT}( T_{(v,u)}\setminus \{u\})+\mathrm{OPT}(T_{(u,v)} \setminus \{v\}) > \mathrm{OPT}(T_{(v,u)}) + \mathrm{OPT}(T_{(u,v)})\]
which can be rearranged into contributions of the left and right parts:
\begin{align}\label{dimer-criterion1}
w(e)>  \bigg(\mathrm{OPT}(T_{(v,u)})- \mathrm{OPT}( T_{(v,u)}\setminus \{u\})\bigg)+\bigg(\mathrm{OPT}(T_{(u,v)})-\mathrm{OPT}(T_{(u,v)} \setminus \{v\})\bigg). 
\end{align}

We thus set the messages on the directed edges $(v,u)$ and $(u,v)$ as follows (refer to Figure~\ref{fig:Zdef} for an illustration):
\begin{align}
     Z(v,u)&= \bigg(\mathrm{OPT}(T_{(v,u)})- \mathrm{OPT}( T_{(v,u)}\setminus \{u\})\bigg),\label{dimer-criterion2}\\
    Z(u,v)&=\bigg(\mathrm{OPT}(T_{(u,v)})-\mathrm{OPT}(T_{(u,v)} \setminus \{v\})\bigg).\label{dimer-criterion3}
\end{align}
\begin{figure}
    \centering
    \includegraphics[width=0.5\linewidth]{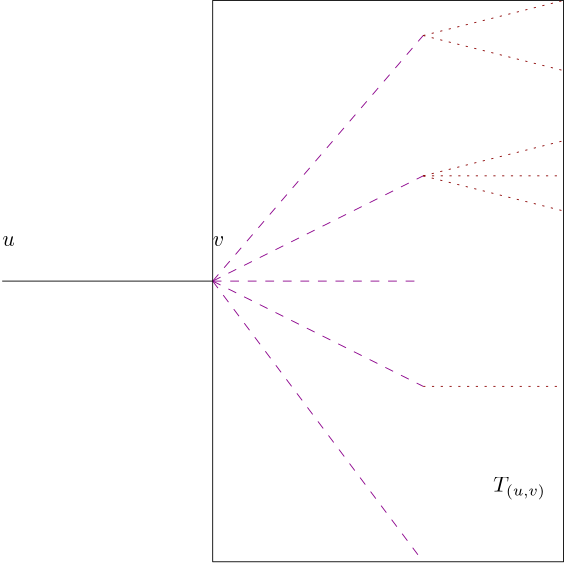}
    \caption{The message $Z(u,v)$ is defined as the difference between the weight of the optimal configuration in the tree in the box and the dotted tree, which is equivalent to the gain of attaching $v$ to the dotted forest through the dashed edges.}
    \label{fig:Zdef}
\end{figure}
From \eqref{dimer-criterion1}--\eqref{dimer-criterion3} we have that the dimer criterion becomes
\begin{equation}\label{eq:edgecriterion}
w(u,v)>Z(u,v)+Z(v,u).
\end{equation}

The message $Z(u,v)$ is exactly the gain of attaching $v$ to the forest $T_{(u,v)} \setminus \{v\}$, which allows us to compute its values recursively on the tree (refer to Figure~\ref{fig:Zrecursion} for an illustration of this derivation).

First, if $v$ is a monomer in $\mathrm{MD}_{\mathrm{opt}}(T_{(u,v)},w,x)$, then the gain is exactly $x(v)$.

Next, assuming that $\{v,u'\}$ is a dimer in the optimal configuration $\mathrm{MD}_{\mathrm{opt}}(T_{(u,v)},w,x)$ for some $u' \sim v$, $u'\neq u$, we gain from the dimer the weight $w(v,u')$, but we have lost the ability to use $u'$ in the subtree $T_{(v,u')}$ instead. The optimal configuration has thus gained $w(v,u')$ and lost $Z(v,u')$ from attaching $\{v\}$, so the total gain is $w(v,u')-Z(v,u')$. 

Finally, the optimal configuration has to optimise among all the possibilities above. Using the convention that the maximum of an empty list is $-\infty$, we get the recursion
\begin{equation}\label{eq:Zrecursion}
     Z(u,v)=\max\left(x(v),\max_{\substack{u' \sim v \\ u' \neq u}}\Big(w(v,u')-Z(v,u')\Big)\right) .  
\end{equation}
\begin{figure}
    \centering
    \includegraphics[width=0.7\linewidth]{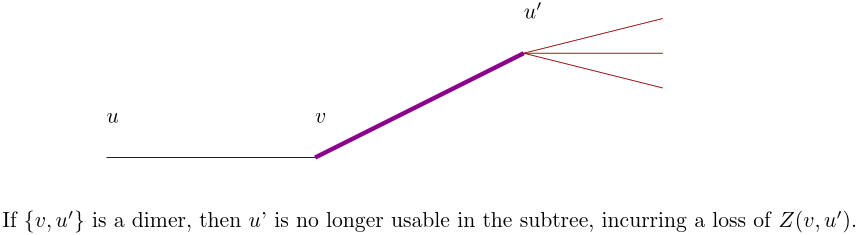}
    \caption{Deduction of the recursion satisfied by the messages.}
    \label{fig:Zrecursion}
\end{figure}

These messages can thus be computed recursively starting from leaves, where they evaluate to the weight of the leaf vertex, and they construct the optimal monomer-dimer configuration in at most $2|E(T)|$ iterations through Criterion~\eqref{eq:edgecriterion}.
The main task in this paper is then to extend this construction and show that there almost surely exists a unique family of messages $Z$ on the weighted rooted UBGW tree $(\mathbb{T},w,x,o)$.

\section{Uniqueness of messages on the limit tree and convergence}
\label{sec:subcritical}
From now on we assume that the conditions (i)-(iii) in Theorem \ref{th:convergenceMD} hold. Recall $m_{\xi}\in \mathbb{R}$ denotes the minimum of the vertex weight distribution $\xi$.

The main arguments in both settings can be broken down as follows:
\begin{enumerate}
    \item First, for a given graph $G$, a vertex $o \in V(G)$ and a depth $H\in \mathbb{N}_0$ such that $B_H(G,o)$ is a tree, we can represent the influence of the optimal monomer-dimer configuration outside the ball $B_H(G,o)$ by a well-chosen boundary condition on the messages on the ball. Namely, take a boundary vertex $v \in \partial B_H(G,o)$ and denote by $u$ its unique neighbour in $B_H(G,o)$. If $v$ is not matched to the exterior of $B_H(G,o)$, then we set $Z(u,v) = x(v)$, but if, on the contrary,  the vertex $v$ is matched to an exterior vertex, then we set $Z(u,v)=+\infty$. Then, the restriction of the optimal monomer-dimer configuration to $B_H(G,o)$ coincides with the monomer-dimer configuration induced by the messages inside $B_H(G,o)$ defined by these boundary conditions.
    \item Another general fact is that inside the ball $B_H(G,o)$, the messages around the root vertex $o$ are deterministically increasing with respect to the boundary conditions when the depth $H$ is even, and decreasing when the depth $H$ is odd. So as $H \rightarrow +\infty$, we get that the messages are almost surely bounded by the limit of two extremal families, one built by setting $Z(u,v)=m_{\xi}$ for $v \in \partial B_H(G,o)$ and $u \in B_H(G,o) $ with $v \sim u$ and the other by setting $Z(u,v)=+\infty$ instead. The main difficulty is to show that the two extremal families converge to the same messages as $H \rightarrow +\infty$. It was only overcome in special cases (mainly when vertex-weights were zero and edge-weights were exponentially distributed) in previous works \cite{gamarnik2003maximum,correlationarnab}, thanks to an explicit formulation of the message distribution. The main contribution of this paper lies in an invariant trick given in Section~\ref{sec:mainprop} that allows us to show that the two extremal families are almost surely equal in i.i.d. weighted UBGW trees $(\mathbb{T},o)$.
    \item As a conclusion, we have that as the depth $H$ increases, any compatible family of messages inside the ball $B_H(G,o)$ must converge to the unique family given by the common almost sure limit of the extremal families. This common limit then constructs the unique ground state of the monomer-dimer model on the limiting tree through Rule~\eqref{eq:edgecriterion}. It is also, by the first point, the universal limit of monomer-dimer models on any graph $(G,o)$ close to the tree $(\mathbb{T},o)$.
\end{enumerate}

Pick a sequence of rooted weighted random graphs $(G_n,w_n,x_n,o_n)$ such that $(G_n,,w_n,x_n,o_n)$ converges locally to $(\mathbb{T},w,x,o)$. 
This is equivalent to saying that almost surely, for any $H\in \mathbb{N}_0$, there exists $N_H\in \mathbb{N}$ such that for all $n>N_H$,
\[ B_H(G_n,w_n,x_n,o_n) \cong B_H(\mathbb{T},w,x,o)  .\]
Our goal is then to show that there exists a locally approximable hence almost surely defined monomer-dimer configuration $\mathbb{MD}$ on $(\mathbb{T},w,x,o)$ such that for every $H_0\in \mathbb{N}_0$, if $\mathrm{MD}(G_n)$ is an optimal monomer-dimer configuration on $(G_n,w_n,x_n)$, then there almost surely exists $N_{H_0}'\in \mathbb{N}$ such that for every $n \geq N'_{H_0}$,
\begin{align}
    B_{H_0}(G_n,\mathrm{MD}(G_n),o_n) \cong B_{H_0}(\mathbb{T},\mathbb{MD},o).    
\end{align}
We will begin by defining $\mathbb{MD}$ and show that it can be characterised as the monomer-dimer configuration given by the almost surely unique family of messages on $(\mathbb{T},o)$ verifying Recursion~\eqref{eq:Zrecursion}. Recall that $m_{\xi}\in \mathbb{R}$ is the minimum of the vertex weight distribution $\xi$. 
\begin{prop}\label{prop:gibbsuniqueness}
Consider the families of truncated decorated trees \[\left\{\left(\mathbb{T}_H,w,x,o,Z_+^H,Z_-^H\right)\right\}_{H\in \mathbb{N}_0}\] defined as follows:
\begin{enumerate}[(1)]
    \item Sample $(\mathbb{T},w,x,o)$ and let $\mathbb{T}_H=B_H(\mathbb{T},w,x,o)$.
    \item Set $Z_-^H=m_{\xi}$ for every edge of $\mathbb{T}_H$ pointing towards the boundary $\partial T_H$.
    \item Define $Z_-^{H}$ on the remainder of $\mathbb{T}_H$, using Recursion~\eqref{eq:Zrecursion}.
    \item Proceed similarly for $Z_{+}^H$ by setting $Z_{+}^H=+\infty$,  instead of $m_\xi$,  on the boundary $\partial T_H$.
    \item Extend both families on $(\mathbb{T},w,x,o)$ by arbitrarily setting any value, say $0$, outside of $\mathbb{T}_H$.
\end{enumerate}

Then $(\mathbb{T},w,x,o)$-almost surely, for any $\overset{\rightarrow}{e} \in \overset{\rightarrow}{E}(\mathbb{T})$, \[Z_+^{H}(\overset{\rightarrow}{e})\] and \[Z_-^{H}(\overset{\rightarrow}{e})\] converge (as $H \rightarrow +\infty$) to a common value $Z(\overset{\rightarrow}{e})$ that is thus $(\mathbb{T},w,x,o)-$measurable.

    Furthermore, $(\mathbb{T},w,x,o)$-almost surely, the only family of messages verifying Recursion~\eqref{eq:Zrecursion} is this common limit $Z$.  
\end{prop}
\begin{proof}[Proof of Proposition~\ref{prop:gibbsuniqueness}]
Let us start by studying messages around the root $o$. We will use the fact that for any $v\sim o$, the random variables $(Z_-^{2H})(o,v)$  (resp. $(Z_+^{2H}) (o,v))$ are deterministically increasing (resp. decreasing) as $H$ increases.
Indeed, for a fixed family of weights $(w_l)_{l \in \mathbb{N}_0}$ and any fixed $x \in \mathbb{R}$, if we pick two families $(a_l)_{l \in \mathbb{N}_0}$ and $(a_l')_{l \in \mathbb{N}_0}$ such that $a_l \leq a_l'$ for every $l\in \mathbb{N}_0$, then for any $N\in \mathbb{N}_0$,
\[ \max \left(x,\max_{0\leq l \leq N} (w_l-a_l)\right) \geq \max\left( x, \max_{0\leq l \leq N} (w_l-a_l')\right). \]
In terms of random variables, the following almost sure limits exist when looking at even depths:
\begin{align*}
    Z_{-}^{\text{even}} &= \lim_{H \rightarrow +\infty} Z_{-}^{2H} ,\\
    Z_{+}^{\text{even}} &= \lim_{H \rightarrow +\infty} Z_{+}^{2H}.
\end{align*}
For the odd depths, the analogous limits also exist around the root:
\begin{align*}
    Z_{-}^{
    \text{odd}} &= \lim_{H \rightarrow +\infty} Z_{-}^{2H+1}, \\
    Z_{+}^{\text{odd}} &= \lim_{H \rightarrow +\infty} Z_{+}^{2H+1}.
\end{align*}
Furthermore, we obtain the following almost sure  chain of inequalities around the root
\begin{align*} 
&Z_{-}^{2H} \leq Z_{+}^{2H+1} \leq Z_{-}^{2H+2} \leq Z_{+}^{2H+2} \leq Z_{-}^{2H+1} \leq Z_{+}^{2H}.    
\end{align*}
Going from the outer edge to the centre, the first inequalities are obtained by looking at the boundary of both variables at depth $2H$, the next inequalities by looking at their boundary at depth $2H+1$, and the centre inequality by looking at their boundary conditions at depth $2H+2$. Taking $H$ to infinity implies that:
\[  Z_{-}^{\text{even}} \leq Z_{+}^{\text{odd}} \leq  Z_{-}^{\text{even}}  \leq Z_{+}^{\text{even}} \leq Z_{-}^{
    \text{odd}} \leq Z_{+}^{\text{even}}, \]
which condenses into
\begin{equation}\label{eq:Zpmeven}
Z_{-}^{\text{even}} = Z_{+}^{\text{odd}} \leq Z_{+}^{\text{even}} = Z_{-}^{
    \text{odd}}.    
\end{equation} 
To show the conclusion, we only need to show that the centre inequality in \eqref{eq:Zpmeven} is in fact an equality. For this purpose, it is sufficient to show that the distribution of the two variables on the left is the same as the distribution of the two variables on the right.
Furthermore for any $v \sim o$, we have that
\[Z_+^{2H+1}(o,v)=\max\bigg(x(v),\max_{\substack{u \sim v \\ u \neq o}} \left( w(v,u)-Z^{2H+1}_+(v,u)   \right) \bigg). \]
Now, the random variable $Z_+^{2H}(o,v)$ has the same distribution as $Z_+^{2H+1}(u,v)$ for any $v \sim u,v \neq o$. Indeed, the subtree that $Z_+^{2H+1}(u,v)$ uses has distribution $(\mathbb{T}_{2H},x,w,o)$ by virtue of the underlying tree being an i.i.d. weighted UBGW and $(u,v)$ also sees a boundary condition valued at $+\infty$.
Writing $\zeta_+^{2H+1}$ to be the distribution of $Z_+^{2H+1}(o,v)$ and $\zeta_+^{2H}$ to be the distribution of $Z_+^{2H}(o,v)$,
we deduce that for any $H\in \mathbb{N}_0$, the following equality in law holds:
\begin{equation}
    Z_+^{2H+1} \overset{\mathcal{L}}{=} \max\bigg( X, \max_{1 \leq l \leq \hat{N}} \Big(w_l-Z_+^{2H, (l)}\Big)\bigg)
\end{equation}
where $Z_+^{2H+1} \sim \zeta_+^{2H+1}$, $X \sim \xi$, $\hat{N} \sim \hat{\pi}$, $(w_l)_{l \in \mathbb{N}}$ an i.i.d. sequence of law $\omega$, and $(Z_+^{2H,(l)})_{l \in \mathbb{N}}$ an i.i.d. sequence of law $\zeta_+^{2H}$, all mutually independent.
We thus deduce, taking $H \rightarrow +\infty$ and applying a symmetrical argument to $Z_-^{2H}$, the distributions of  $Z_-^{\text{even}}$ and $Z_+^{\text{even}}$, that we will denote by $\zeta_-$ and $\zeta_+$, must verify a recursive distributional system that we describe below.  

If $(Z_-^{(l)})_{l \in \mathbb{N}_0}$ is an i.i.d. sequence of distribution $\zeta_-$, $(Z_+^{(l)})_{l \in \mathbb{N}_0}$ is an i.i.d. sequence of distribution $\zeta_+$, $X$ of distribution $\xi$, $(w_l)_{l \in \mathbb{N}_0}$ an i.i.d. sequence of distribution $\omega$, and finally $\hat{N}$ of distribution $\hat{\pi}$, all mutually independent, then the two following equalities hold in law:
\begin{align}
    Z_-^{(0)} &\overset{\mathcal{L}}{=} \max\left(X,\max_{1\leq l \leq \hat{N}} \Big(w_l-Z_+^{(l)}\Big)\right), \label{eq:Z+RDE} \\
    Z_+^{(0)} &\overset{\mathcal{L}}{=} \max\left(X,\max_{1\leq l \leq \hat{N}} \Big(w_l-Z_-^{(l)}\Big)\right). \label{eq:Z-RDE}
\end{align}

Let us denote by $h_+$ and $h_-$ the cumulative distribution functions corresponding to $\zeta_+$ and $\zeta_-$. Note that from \eqref{eq:Zpmeven}, more precisely,  from $Z_-^{\text{even}} \leq Z_+^{\text{even}}$  we have $h_+\leq h_-$. Letting $h_X$ be the cumulative distribution function corresponding to $\xi$ and $\hat{\phi}$ be the probability generating function of $\hat{\pi}$, we can rewrite the system~\eqref{eq:Z+RDE}--\eqref{eq:Z-RDE} as the recursive distributional equations (RDE):
\begin{align}
    h_+(t)&=h_X(t)\hat{\phi}\Big( 1-\mathbb{E}_{W\sim \omega} [h_-(W-t) ]  \Big), \label{eq:h+RDE}  \\
    h_-(t)&=h_X(t)\hat{\phi}\Big( 1-\mathbb{E}_{W\sim \omega} [h_+(W-t)  ]  \Big). \label{eq:h-RDE}
\end{align}

Our next task will be to show that this relation is sufficient to guarantee that $h_-$ and $h_+$ are equal. This is the heart of the paper and will be established in Section~\ref{sec:mainprop}.

\begin{prop}\label{prop:mainprop}
For any pair  $(h_+,h_-)$ of cumulative distribution functions  with $h_+ \leq h_-$ that verify Equations~\eqref{eq:h+RDE}--\eqref{eq:h-RDE}, we have  $h_-=h_+$.
\end{prop}
We continue the proof of Proposition~\ref{prop:gibbsuniqueness} by assuming that Proposition~\ref{prop:mainprop} holds, in other words, we have $h_+=h_-$ and hence $\zeta_+=\zeta_-$. Since the inequality $Z_+\geq Z_-$ holds almost surely, we must have $Z_+=Z_-$ almost surely.
Finally, the reasoning holds on any finite neighbourhood of the root, with the difference that the upper bound and lower bounds $Z_+^{\mathrm{even}}(\overset{\rightarrow}{e}),Z_+^{\mathrm{odd}}(\overset{\rightarrow}{e})$ flips based on the parity of the depth of the oriented edge $\overset{\rightarrow}{e}$. 
Next, we will show that this is the unique family satisfying Recursion~\eqref{eq:Zrecursion}. Let $Z$ be any family verifying Recursion~\eqref{eq:Zrecursion} on $(\mathbb{T},w,x,o)$. Fix some $H\in \mathbb{N}_0$. Since the values of this family on the $H-$boundary lie between $m_{\xi}$ and $+\infty$,  we have that $Z$ lies between $Z_-^{H}$ and $Z_+^{H}$. Since we have just shown that $Z_-^{H}$ and $Z_+^{H}$ converge almost surely to the same limit $Z_-=Z_+$ as $H$ goes to infinity, taking $H \rightarrow +\infty$ yields that $Z=Z_-=Z_+$ almost surely.

\end{proof}

The next argument, contained in the lemma below, makes precise the link between the optimal monomer-dimer configuration on a finite tree-like ball of a graph and a family of suitable messages with appropriate boundary conditions. It will be used to establish the link between $(\mathbb{T},w,x,o)$ and $(G_n,w_n,x_n)$ in the proof of Theorem~\ref{th:convergenceMD}. 
\begin{lemma}\label{lem:matchingbord}
    Let $(G,w,x,o)$ be a finite rooted weighted graph. Let $H\in \mathbb{N}_0$ and assume that $G_H:=B_H(G,w,x,o)$ is a tree with a unique optimal monomer-dimer configuration. Let $\mathrm{MD}_{\mathrm{opt}}(G)$ be any optimal monomer-dimer configuration on $(G,w,x)$. Consider the decorated graph $(G_H,o, Z_{\mathrm{MD}}^{H})$ where $Z_{\mathrm{MD}}^{H}$ is defined as follows:
    \begin{enumerate}
        \item For $u\sim v$ with $u\in G_H$ and $v\in \partial G_H$, set
    \begin{itemize}
        \item $Z_{\mathrm{MD}}^{H}(u,v)=x(v)$ if $v$ is not matched to the exterior of $G_H$ by $\mathrm{MD}_{\mathrm{opt}}(G)$, in other words, $\{v,u'\}$ is not a dimer in $\mathrm{MD}_{\mathrm{opt}}(G)$ for any  vertex $u' \in G\setminus G_H$ with $u' \sim v $;
        \item $Z_{\mathrm{MD}}^{H}(u,v)=+\infty$ otherwise. 
    \end{itemize}
    
       \item Use Recursion~\eqref{eq:Zrecursion} to define $Z_{\mathrm{MD}}^{H}$ on the remainder of $(G_H,o)$.
    \end{enumerate}
   Then any edge $\{u,v\}\in G_H$ is a dimer in $\mathrm{MD}_{\mathrm{opt}}(G)$ if and only if 
     \begin{align*}
        w(u,v) > Z_{\mathrm{MD}}^{H}(u,v)+Z_{\mathrm{MD}}^{H}(v,u).
     \end{align*}
\end{lemma}
\begin{proof}[Proof of Lemma~\ref{lem:matchingbord}]
The idea is that the optimal monomer-dimer configuration on $G_H$ behaves as if every vertex on the boundary matched to the exterior in $\mathrm{MD}_{\mathrm{opt}}(G)$ is removed and untouched otherwise. Setting $Z_{\mathrm{MD}}^{H}(u,v)=x(v)$ has the same effect in $G_H$ as if the vertex $v$ were a leaf. Similarly, setting $Z_{\mathrm{MD}}^{H}(u,v)=+\infty$ forbids the edge $\{u,v\}$ to be a dimer. Indeed, we always have  $+\infty> w(u,v)$ and this removes any effect that the edge $\{u,v\}$ has in the rest of $G_H$ since the value appearing in Recursion~\eqref{eq:Zrecursion} in the next generation from $\{u,v\}$ is $w(u,v)-(+\infty)=-\infty$ which is always smaller than $x(v)$. We know that the $Z$ message formalism is exact on finite trees, so the rule    
\[ w(u,v) > Z_{\mathrm{MD}}^{H}(u,v)+Z_{\mathrm{MD}}^{H}(v,u) \] 
constructs an optimal monomer-dimer configuration on 
\[G^{H,\star}:=G_H \setminus \big\{v \in G: \left(\exists u' \in G\setminus G_H\right) \wedge \big(\{v,u'\} \text{ is a dimer in } \mathrm{MD}_{\mathrm{opt}}(G) \big) \big\}.\] 
\end{proof}

We can now combine Proposition~\ref{prop:gibbsuniqueness} and Lemma~\ref{lem:matchingbord} to prove Theorem~\ref{th:convergenceMD}.
\begin{proof}[Proof of Theorem~\ref{th:convergenceMD}.]
Fix $H_0\in \mathbb{N}_0$.
Our goal is to show that 
\[ \mathbb{P}\bigg( \exists N >0 \, , \, \forall n > N, B_{H_0}(G_n,\mathrm{MD}_{\mathrm{opt}}(G_n),o_n) \cong B_{H_0}(\mathbb{T},\mathbb{MD},o)  \bigg)  =1 .\]
To this end, fix any arbitrary $\eps>0$, we will show that there exists $N_{\eps,H_0}\in \mathbb{N}$ such that
\begin{equation}\label{ed:goodprob}
\mathbb{P}\bigg( \forall n > N_{\eps,H_0}, B_{H_0}(G_n,\mathrm{MD}_{\mathrm{opt}}(G_n),o_n) \cong B_{H_0}(\mathbb{T},\mathbb{MD},o)  \bigg)>1-\eps.
\end{equation}

First, fix $\eps',\eps''>0$, let us first localise $B_{H_0}(\mathbb{T},o)$ by picking $\Delta>0$ large enough so that the event $$\Omega_{\Delta,H_0}:=\{ B_{H_0}(\mathbb{T},o) \text{ has degrees uniformly bounded above by }\Delta \}$$ has probability at least $1-\eps''$. In this case, the ball is uniformly bounded. We can then use Egorov's Theorem on each edge $\{u,v\}$ in Proposition~\ref{prop:gibbsuniqueness} to find an event $\Omega_{\mathrm{good},H_0}$ of probability at least $1-\eps''$ on which there exists $H\in \mathbb{N}$ large enough such that simultaneously for every $\{u,v\} \in \mathbb{T}_{H_0}$, 
\begin{equation}\label{eq:Z+-uniformcontrol}
 \left|Z_-^{H_0+H}(u,v)+Z_+^{H_0+H}(u,v)\right| < \eps'.
\end{equation}
Using local convergence of $(G_n,w_n,x_n,o_n)$, there exists $N_{H_0+H}\in \mathbb{N}$ such that for all $n>N_{H_0+H}$,
\[B_{H_0+H}(G_n,w_n,x_n,o_n) \cong B_{H_0+H}(\mathbb{T},w,x,o).\]
Take any optimal monomer-dimer configuration $\mathrm{MD}_{\mathrm{opt}}(G_n)$ on $G_n$. We know that \[B_{H_0+H}(G_n,w_n,x_n,o_n)\] is isomorphic to a tree with i.i.d. weights, so there is a unique optimal monomer-dimer configuration on 
\[B_{H_0+H}(\tilde{G}_n,\tilde{o}_n) \setminus \{v\in G: \exists u' \in \tilde{G}_n \setminus B_{H_0+H}(\tilde{G}_n,\tilde{o}_n), \{v,u'\} \text{ is a dimer in } \mathrm{MD}_{\mathrm{opt}}(G_n)\}.\]

Using Lemma~\ref{lem:matchingbord}, $\mathrm{MD}_{\mathrm{opt}}(G_n)$ agrees on $B_{H_0+H}(G_n,o_n)$ with a family of messages $Z_{\mathrm{MD}}^{H_0+H}$ constructed with boundary conditions specified inside the lemma. 

We can now use the isomorphism between $B_{H_0+H}(G_n,w_n,x_n,o_n)$ and $B_{H_0+H}(\mathbb{T},w,x,o)$ to map $Z_{\mathrm{MD}}^{H_0+H}$ on $B_{H_0+H}(\mathbb{T},w,x,o)$ into a family $Z_{\mathrm{MD}}^{H_0+H}(\mathbb{T)}$ verifying Recursion~\eqref{eq:Zrecursion} on $B_{H_{0}+H}(\mathbb{T})$. We can once again use the almost sure decreasing property of the messages to obtain that this family $Z_{\mathrm{MD}}^{H+H_0}(\mathbb{T})$ must be contained in the interval between $Z_-^{H_0+H}$ and $Z_+^{H_0+H}$:
\[  Z_{\mathrm{MD}}^{H_0+H}(\mathbb{T}) \in \left[\min\left(Z_-^{H_0+H},Z_{+}^{H_0+H}\right) , \max \left( Z_-^{H_0+H},Z_+^{H_0+H}\right) \right].  \]
Since $Z$ on $(\mathbb{T},w,x,o)$ also lies in the above interval and Inequality~\eqref{eq:Z+-uniformcontrol} implies that the length of this interval is uniformly bounded by $\eps'$, we obtain the uniform bound on valid on $\mathbb{T}_H$:
\[  \left|Z-Z_{\mathrm{MD}}^{H_0+H}(\mathbb{T)}\right| <\eps'.   \]
This allows us to state that the random sets   
\[ \left\{w(u,v) > Z_{\mathrm{MD}}^{H_0+H}(\mathbb{T})(u,v)+Z_{\mathrm{MD}}^{H_0+H}(\mathbb{T})(v,u) \right\} \]
and 
\[ \left\{w(u,v) > Z(u,v)+Z(v,u) \right\}  \]
on $B_{H_0}(\mathbb{T},o)$ can only differ if there exists an edge $\{u,v\} \in B_{H_0}(\mathbb{T},o)$ such that $w(u,v)$ lies in the interval between
\[ \min \left( Z(u,v)+Z(v,u),Z_{\mathrm{MD}}^{H_0+H}(\mathbb{T})(u,v)+Z_{\mathrm{MD}}^{H_0+H}(\mathbb{T})(v,u)  \right) \] and\[ \max \left( Z(u,v)+Z(v,u),Z_{\mathrm{MD}}^{H_0+H}(\mathbb{T})(u,v)+Z_{\mathrm{MD}}^{H_0+H}(\mathbb{T})(v,u)\right) .\]
Inequality~\eqref{eq:Z+-uniformcontrol} implies that the size of each of these intervals is bounded by at most $2\eps'$.

Let us write $A_{\mathrm{bad}}(u,v)$ to be the corresponding event  for each edge $\{u,v\} \in B_{H_0}(\mathbb{T},o)$.
Since the random set 
\[\left\{w(u,v) > Z_{\mathrm{MD}}^{H_0+H}(\mathbb{T})(u,v)+Z_{\mathrm{MD}}^{H_0+H}(\mathbb{T})(v,u) \right\} \] identifies with $\mathrm{MD}_{\mathrm{opt}}(G_n)$, we have shown that the event
\[ \bigg(\left\{ \forall n > N_{H_0}, B_{H_0}(G_n,\mathrm{MD}_{\mathrm{opt}}(G_n),o_n) \cong B_{H_0}(\mathbb{T},\mathbb{MD},o)  \right\}    \cap \Omega_{\Delta,H_0} \cap  \Omega_{\mathrm{good},H_0}   \bigg)^{\complement}  \]
is contained in the event
\[ \bigcup_{ \{u,v\} \in B_{H_0}(\mathbb{T},o)}A_{\mathrm{bad}}(u,v).   \]

Let us define the map $g_\omega: t_0 \mapsto \sup_{t \in \mathbb{R}} \mathbb{P}_{W \sim \omega}(W \in [t,t+t_0]) $. We have that $g_\omega(0)=0$ and $g_\omega$ is continuous around $0$ as $\omega$ is continuous. Conditioning on the unweighted configurations appearing in $B_{H_0}(\mathbb{T},o)$ and using the fact that the weights are independent from the unweighted tree, we obtain 
\[ \mathbb{P}\left(\bigcup\nolimits_{ \{u,v\} \in B_{H_0}(\mathbb{T},o)}A_{\mathrm{bad}}(u,v)  \right)  \leq   \mathbb{E}[ |E(B_{H_0}(\mathbb{T},o)| ]\ g_\omega(2\eps') .\]
Since $\hat{\pi}$ is assumed to have finite expectation, $\mathbb{E}[ |E(B_{H_0}(\mathbb{T},o)| ]$ is finite for every $H_0\in \mathbb{N}_0$.
Putting everything together, using the fact that both $\Omega_{\Delta,H_0}$ and $\Omega_{\mathrm{good},H_0}$ have probability at least $1-\eps''$, we have shown that
\[ \mathbb{P}\bigg(\forall n > N_{H+H_0}, B_{H_0}(G_n,\mathrm{MD}_{\mathrm{opt}}(G_n),o_n) \cong B_{H_0}(\mathbb{T},\mathbb{MD},o) \bigg) \]
is lower bounded by
\[ 1- \mathbb{E}[|E(B_{H_0}(\mathbb{T},o)| ]\ g_\omega(2\eps')-2\eps'' . \]
By choosing $\eps',\eps''$ such that $\mathbb{E}[ |E(B_{H_0}(\mathbb{T},o)| ]\ g_{\omega}(2\eps')+2\eps''<\eps$, we obtain \eqref{ed:goodprob}, finishing the proof.
\end{proof}

\begin{remark}
        If one would write this proof for $L^1$-local topology of weighted graphs, then the interval in the condition on each $A_{\mathrm{bad}}(u,v)$ would grow by an additional $\eps'$ incurred from the allowed variability of the weights in $B_{H_0+H}(G_n,w_n,x_n,o_n)$ compared to $B_{H_0+H}(\mathbb{T},w,x,o)$. Indeed, the families $Z_+^{H_0+H}$ and $Z_{-}^{H_0+H}$ are $1-$Lipschitz with respect to every weight appearing in $B_{H_0+H}(G_n,w_n,x_n,o_n)$. The proof would thus be the same with $g_{\omega}(3\eps')$ instead of $g_{\omega}(2\eps')$.
\end{remark}
We now explain the adaptation of the proof of Theorem~\ref{th:convergenceMD} to that of Theorem~\ref{th:decay}.
\begin{proof}[Proof of Theorem~\ref{th:decay}]
    This theorem is essentially proved the same way, except that instead of considering one ball, we consider two balls around $o_n$ and $o_n'$. By using the condition that the graph distance between $o_n$ and $o_n'$ goes to infinity, for any $H,H'\in \mathbb{N}_0$, we can find $N_{H,H'}\in \mathbb{N}$ large enough so that for any $n\geq N_{H,H'}$, the two balls $B_{H}(G_n,o_n)$ and $B_{H'}(G_n,o_n')$ do not intersect. We can then continue by writing that the influence of one ball on the other is always uniformly controlled by the extremal families in each ball, which maps to extremal families of two independent weighted unimodular Bienaymé-Galton-Watson trees by applying the joint convergence hypothesis. The remainder follows in the same way.
\end{proof}

\begin{proof}[Proof of Corollary~\ref{coro:averageweight}]
Its deduction  is standard once correlation decay is established, as it implies that the covariance between the states of two edges in $\mathrm{MD}_{\mathrm{opt}}(G_n)$ decays with their distance $d$ at the speed with which $Z_+^d$ and $Z_-^d$ converge to their common limit. We thus omit the details.
\end{proof}

\section{Invariance trick}\label{sec:mainprop}
In this section, we prove Proposition~\ref{prop:mainprop}. 
Recall the RDE
\begin{align}
    h_+(t)&=h_X(t)\hat{\phi}\Big( 1-\mathbb{E}_{W\sim \omega} [h_-(W-t) ]  \Big), \tag{\ref{eq:h+RDE}}  \\
    h_-(t)&=h_X(t)\hat{\phi}\Big( 1-\mathbb{E}_{W\sim \omega} [h_+(W-t)  ]  \Big). \tag{\ref{eq:h-RDE}}
\end{align}

The trick is to introduce the following quantity that will allow us to define an invariant:
\begin{equation}
     \mathbb{P}\left(  Z_-+Z_+<W   \right)
\end{equation}
where $Z_+,Z_-,W$ are three independent random variables of distribution $\zeta_+,\zeta_-,\omega$, respectively.

Recall that $h_+$ and $h_-$ denote the cumulative distribution functions corresponding to $\zeta_+$ and $\zeta_-$, respectively.    
By conditioning either over the value of $Z_+$ or over the value of $Z_-$, we get the following equality:
\begin{equation}\label{eq:invariant}
    \int_{\mathbb{R}} \mathbb{P}(Z_-<W-t)\mathrm{d}h_+(t) = \int_{\mathbb{R}}\mathbb{P}(Z_+<W-t)\mathrm{d}h_-(t).
\end{equation}
Rewritten in terms of expectations, Equality \eqref{eq:invariant} yields
\begin{equation}\label{eq:expectaton}
    \int_{\mathbb{R}} \mathbb{E}_{W\sim \omega}[h_-(W-t)  ]\mathrm{d}h_+(t)=\int_{\mathbb{R}}\mathbb{E}_{W\sim \omega}[h_+(W-t)]\mathrm{d}h_-(t).
\end{equation}
We naturally have that $h_+=h_-=0$ on $(-\infty,m_{\xi})$, so using RDE~\eqref{eq:h+RDE}--\eqref{eq:h-RDE}, the fact  that $h_X>0$ on $(m_{\xi},+\infty)$, and the invertibility of the probability generating function $\hat{\phi}$ of $\hat{\pi}$,   we obtain from Equality \eqref{eq:expectaton} that
\begin{equation}\label{eq:inverse-pgf}
    \int_{m_{\xi}}^{+\infty}  \left( 1-\hat{\phi}^{-1}\left(\frac{h_+(t)}{h_X(t)}   \right)    \right) \mathrm{d}h_+(t)=\int_{m_{\xi}}^{+\infty}  \left( 1-\hat{\phi}^{-1}\left(\frac{h_-(t)}{h_X(t)}   \right)    \right) \mathrm{d}h_-(t). 
\end{equation}
Using the function $F:\, u \mapsto (1-\hat{\phi}^{-1}(u))$ that is decreasing on $[0,1]$, we can rewrite Equality \eqref{eq:inverse-pgf} as:
\begin{equation} \label{eq:masterequation}
    \int_{m_{\xi}}^{+\infty} F\left(\frac{h_+(t)}{h_X(t)}\right) \mathrm{d}h_+(t)= \int_{m_{\xi}}^{+\infty} F\left(\frac{h_-(t)}{h_X(t)}\right) \mathrm{d}h_-(t).        
\end{equation}

This quantity, which can be written as a function of distribution $\zeta_+$ or $\zeta_-$, does not depend on which of the distributions we choose and is thus an invariant over solutions to RDE~\eqref{eq:h+RDE}--\eqref{eq:h-RDE}.
We will now use the fact $Z_-^{\text{even}} \leq Z_+^{\text{even}}$, which implies $$h_+\leq h_-.$$

In the remainder of the proof we will consider first the case where $\xi$ is continuous on $[m_{\xi},+\infty)$, and then the case allowing $\xi$ to have an atom at $m_{\xi}$. 

We denote by $S_{\omega}$, $S_{\xi}$, $S_{\zeta_+}$ and $S_{\zeta_-}$ the support of $\omega$, $\xi$, $\zeta_+$ and $\zeta_-$, respectively.

\subsection{Case with \texorpdfstring{$\xi$}{xi} continuous}\label{case:non-atomic}
Assume $\xi$ is continuous on $[m_{\xi},+\infty)$. Then the right-hand side of Equations~\eqref{eq:h+RDE}--\eqref{eq:h-RDE} are both continuous by virtue of both $\xi,\omega$ being continuous, so $\zeta_+$ and $\zeta_-$ are also continuous.

Writing $h_+^{-1}$ and $h_-^{-1}$ to be the generalised inverses of $h_+$ and $h_-$, we can then set the change of variable $h_+(t)=u,h_-(t)=u$ in Equation~\eqref{eq:masterequation} to obtain:

\begin{equation}\label{eq:masterequationcontinuous}
    \int_0^{1}F\left( \frac{u}{ h_X(h_+^{-1}(u))  } \right) \mathrm{d}u=\int_0^{1}F\left( \frac{u}{ h_X(h_-^{-1}(u))  } \right) \mathrm{d}u.
\end{equation}

Using the facts that $F$ is decreasing, $h_X$ is non-decreasing, and $h_+ \leq h_-$ (and thus $h_+^{-1} \geq h_-^{-1}$), we deduce that
\begin{equation}\label{eq:F-inequality}
F\left( \frac{u}{ h_X(h_+^{-1}(u))  } \right) \geq  F\left( \frac{u}{ h_X(h_-^{-1}(u))  } \right) 
\end{equation}
with equality if and only if 
\[ h_X(h_+^{-1}(u))= h_X(h_-^{-1}(u)) .\]

Since Equation~\eqref{eq:masterequationcontinuous} implies that the integrals of both sides of \eqref{eq:F-inequality} are equal on $[0,1]$, we deduce that  
\begin{equation} \label{eq:h_x}  
h_X \circ h_+^{-1}= h_X \circ h_-^{-1}  \quad \text{a.e.} \quad \text{ on } \quad [0,1].
\end{equation}

Recall that $h_X$ denotes the cumulative distribution function of $X\sim \xi$. Here there is a case distinction depending on whether $\xi$ has flat pieces. For simplicity sake, we will assume that $S_{\xi}=[m_{\xi},+\infty)$ in this section and defer the technical general case to Appendix~B.
If it is such, then $h_X$ is invertible on $[m_{\xi},+\infty]$ which contains $S_{\zeta_+}$ and $S_{\zeta_-}$, hence
\begin{equation*}
    h_+^{-1}=h_-^{-1} \quad \text{a.e.} \quad \text{ on} \quad [0,1].
\end{equation*}
As quantile functions uniquely determine the distribution of a random variable, we have thus shown that $h_+=h_-$.
\subsection{Case \texorpdfstring{$\xi$}{xi} with a singular atom at its minimum}\label{case:atomic}

We now assume that $\xi$ is continuous outside of an atom of mass $\alpha>0$ located at the minimum $m_\xi$ of $\xi$.

Once again, RDE~\eqref{eq:h+RDE}--\eqref{eq:h-RDE} show that $\zeta_+,\zeta_-$ are continuous outside of a potential atom located at $m_\xi$.

We will show that
$h_+(m_\xi)=h_-(m_\xi)$. We start with the following proposition that guarantees that both are positive, whose proof with elementary analysis by adapting \cite[Lemma 3.3]{enriquez2024optimalunimodularmatching} is deferred to Appendix~A.
\begin{prop}\label{prop:atomic}
    Assume that $\xi$ is atomic at its minimum $m_\xi$ and $\omega$ is continuous. Then any pair $(h_+,h_-)$ solution to RDE~\eqref{eq:h+RDE}--\eqref{eq:h-RDE} verifies that both $h_+$ and $h_-$ are atomic at $m_\xi$. 
\end{prop}

Set 
\begin{equation*}
\beta_+=h_+(m_\xi) \quad \text{and} \quad \beta_-=h_-(m_\xi).    
\end{equation*}
Note that $\beta_+ \leq \beta_-$ because $h_+ \leq h_-$ and that Proposition~\ref{prop:atomic} implies $\beta_+,\beta_->0$.

Define $\mu_-$ (resp. $\mu_+$) to be the pushforward measure of $\mathrm{d}h_-$ by $h_-$ (resp. $\mathrm{d}h_+$ by $h_+$). Analogously to Equation~\eqref{eq:masterequationcontinuous}, we have
\[     \int_{\mathbb{R}} F\left(\frac{u}{h_X( h_+^{-1}(u) )}\right) \mathrm{d}\mu_+(u)= \int_{\mathbb{R}} F\left(\frac{u}{h_X( h_-^{-1}(u) )}\right) \mathrm{d}\mu_-(u).           \]

As $h_+$ and $h_-$ are continuous on $(m_\xi,+\infty)$ with respective image $(\beta_+,1)$ and $(\beta_-,1)$, we have that $\mathrm{d}\mu_+(u)=\mathrm{d}u$ on $(\beta_+,1)$ and similarly $\mathrm{d}\mu_-(u)=\mathrm{d}u$ on $(\beta_-,1)$. Taking into account the atom at $m_\xi$ of respective mass $\beta_+$ and $\beta_-$ gives:
\begin{align*}
     &\int_{\beta_+}^1 F\left(\frac{u}{h_X( h_+^{-1}(u) )}\right) \mathrm{d}u + \beta_+\,  F\left(\frac{\beta_+}{h_X( h_+^{-1}(\beta_+) )}\right)     \\ =&\int_{\beta_-}^1 F\left(\frac{u}{h_X( h_-^{-1}(u) )}\right)  \mathrm{d}u  +  \beta_-\,  F\left(\frac{\beta_-}{h_X( h_-^{-1}(\beta_-) )}\right).    
\end{align*}
By definition, we have $ h_+^{-1}(\beta_+) =m_\xi$ and $h_-^{-1}(\beta_-) =m_\xi$, and by assumption, $ h_X(m_\xi)=\alpha$, so the above equality becomes
\begin{equation}\label{eq:masterequationminatom}
 \int_{\beta_+}^1 F\left(\frac{u}{h_X( h_+^{-1}(u) )}\right) \mathrm{d}u + \beta_+\, F\left(\frac{\beta_+}{\alpha}\right)   =\int_{\beta_-}^1 F\left(\frac{u}{h_X( h_-^{-1}(u) )}\right)  \mathrm{d}u  +  \beta_-\,  F\left(\frac{\beta_-}{\alpha}\right). 
\end{equation}

As in Case \ref{case:non-atomic}, we use Inequality \eqref{eq:F-inequality}, now valid for  $ u \in [\beta_-,1] $,
\[ F\left( \frac{u}{ h_X(h_+^{-1}(u))  } \right) \geq  F\left( \frac{u}{ h_X(h_-^{-1}(u))  } \right) \]
with equality if and only if 
\[ h_X(h_+^{-1}(u))= h_X(h_-^{-1}(u)) .\]
Integrating this inequality over $ u \in [\beta_-,1] $ yields
\begin{equation}\label{eq:intermediate1}
  \int_{\beta_-}^1 F\left( \frac{u}{ h_X(h_+^{-1}(u))  } \right) \mathrm{d}u \geq  \int_{\beta_-}^1 F\left( \frac{u}{ h_X(h_-^{-1}(u))  } \right) \mathrm{d}u .
\end{equation}

Using the facts that $h_+^{-1} \geq h_-^{-1}$ and $h_X$ is non-decreasing, we have that for any $u \in [\beta_+,\beta_-]$,  $h_X(h_+^{-1}(u))  \geq h_X(h_+^{-1}(\beta_+)) = h_X(m_\xi)=\alpha$ and thus $\frac{u}{h_X(h_+^{-1}(u))} \leq \frac{u}{\alpha}$. Since $F$ is decreasing, we deduce that 
\[ F\left(\frac{u}{h_X(h_+^{-1}(u))}\right) \geq F\left( \frac{u}{\alpha}  \right)   \quad \text{for} \quad u \in [\beta_+,\beta_-].\]
Integrating this inequality over $[\beta_+,\beta_-]$ gives:
\begin{equation}\label{eq:intermediate2}
 \int_{\beta_+}^{\beta_-} F\left(\frac{u}{h_X(h_+^{-1}(u))}\right) \mathrm{d}u \geq \int_{\beta_+}^{\beta_-} F\left( \frac{u}{\alpha}  \right) \mathrm{d}u.
\end{equation}

Integrating the right-hand side of Inequality \eqref{eq:intermediate2} by part, and remembering that $F:u \mapsto 1-\hat{\phi}^{-1}(u)$, and hence $F':u \mapsto -(\hat{\phi}^{-1})'(u) = \frac{-1}{\hat{\phi}'(\hat{\phi}^{-1}(u))} < 0$, yields
\begin{align*}
\int_{\beta_+}^{\beta_-} F\left( \frac{u}{\alpha}  \right)\mathrm{d}u
&= \left[ uF\left(\frac{u}{\alpha}\right)   \right]_{\beta_+}^{\beta_-} -\int_{\beta_+}^{\beta_-} \frac{u}{\alpha}F'\left(\frac{u}{\alpha}\right)\mathrm{d}u  \\
&= \beta_-\, F\left( \frac{\beta_-}{\alpha}\right) - \beta_+\, F\left(\frac{\beta_+}{\alpha} \right) -\int_{\beta_+}^{\beta_-} \frac{u}{\alpha} F'\left( \frac{u}{\alpha}\right)\mathrm{d}u \\
&\geq \beta_-\, F\left( \frac{\beta_-}{\alpha}\right) - \beta_+\, F\left(\frac{\beta_+}{\alpha} \right)
\end{align*}
with equality occurring if and only if $\beta_-=\beta_+$ as $F'$ is non-zero on $[0,1]$.
Putting this together with Inequality~\eqref{eq:intermediate2} gives

\begin{equation}\label{eq:intermediate3}
     \int_{\beta_+}^{\beta_-} F\left(\frac{u}{h_X(h_+^{-1}(u))}\right) \mathrm{d}u  \geq \beta_-\, F\left( \frac{\beta_-}{\alpha}\right) - \beta_+\, F\left(\frac{\beta_+}{\alpha} \right).
\end{equation}

Putting Inequality~\eqref{eq:intermediate1} and Inequality~\eqref{eq:intermediate3} together gives

\begin{equation}\label{intermediate4}
     \int_{\beta_+}^1 F\left( \frac{u}{ h_X(h_+^{-1}(u))  } \right) \mathrm{d}u \geq  \int_{\beta_-}^1 F\left( \frac{u}{ h_X(h_-^{-1}(u))  } \right) \mathrm{d}u + \beta_-\, F\left( \frac{\beta_-}{\alpha}\right) - \beta_+\, F\left(\frac{\beta_+}{\alpha} \right),
\end{equation}
which is equivalent to 
\[      \int_{\beta_+}^1 F\left( \frac{u}{ h_X(h_+^{-1}(u))  } \right) \mathrm{d}u + \beta_+\, F\left(\frac{\beta_+}{\alpha} \right) \geq  \int_{\beta_-}^1 F\left( \frac{u}{ h_X(h_-^{-1}(u))  } \right) \mathrm{d}u + \beta_-\, F\left( \frac{\beta_-}{\alpha}\right)    \]
with equality implying $\beta_-=\beta_+$ and $h_X\circ h_+^{-1}=h_X\circ h_-^{-1}$ almost everywhere on $[\beta_-,1]$.

But the previous deduced Equation~\eqref{eq:masterequationminatom} states precisely that the equality holds, and we thus deduce that $\beta_+=\beta_-$ which translates by definition into 
\begin{align}\label{eq:equality_at_minimum}
 h_+(m_\xi)=h_-(m_\xi).   
\end{align}

Now use RDE~\eqref{eq:h+RDE}--\eqref{eq:h-RDE} to expand this equality into

\[ h_X(m_{\xi})\hat{\phi}\Big(1-\mathbb{E}_{W\sim \omega}\left[ h_-(W-m_{\xi})\right]\Big)=h_X(m_{\xi})\hat{\phi}\Big(1-\mathbb{E}_{W\sim \omega}\left[ h_-(W-m_{\xi})\right]\Big) . \]
As $h_X(m_{\xi})=\alpha>0$, we can cancel out on both sides and use invertibility of $\hat{\phi}$ to obtain
\[ \mathbb{E}_{W\sim \omega}\left[ h_-(W-m_{\xi})\right]=\mathbb{E}_{W\sim \omega}\left[ h_+(W-m_{\xi})\right]\]
and then use that $h_+ \geq h_-$ to get that $W-$a.e.
\begin{align} 
h_-(W-m_{\xi})=h_+(W-m_{\xi}). \label{eq:hpm}
\end{align}

Note that if $M_{\omega}-m_{\xi} \leq  m_{\xi}$, then $h_+=h_-=h_X$, equivalently $\zeta_+=\zeta_-=\xi$, are clearly the (maximal and minimal) solutions to RDE~\eqref{eq:h+RDE}--\eqref{eq:h-RDE}. Thus we assume   $M_{\omega}-m_{\xi} > m_{\xi}$ below and in Appendix~B.

Now, if $S_{\omega}$ spanned $[m_{\xi},M_{\omega}]$ and $m_{\xi}\geq 0$, then Equation~\eqref{eq:hpm} implies by density that $h_-=h_+$ on $[m_{\xi},M_{\omega}-m_{\xi}]$. Furthermore, RDE~\eqref{eq:h+RDE}--\eqref{eq:h-RDE}  imply that $h_+=h_-=h_X$ on $(M_{\omega}-m_{\xi},+\infty)$. So we are done under condition that $S_{\omega}$ contains $[m_{\xi},M_{\omega}]$ and $m_{\xi}\geq 0$. Once again, we will defer the technical general case where flat pieces appear to Appendix~B.
 \bigskip    
\addcontentsline{toc}{section}{References}
\printbibliography

\bigskip 

\appendix

\section*{Appendix A}\label{Appendix_Prop5.3}

\begin{proof}[Proof of Proposition~\ref{prop:atomic}]
    Since $h_+ \leq h_-$, the statement is equivalent to showing $h_+(m_\xi)>0$. We will reason by contradiction.
    Assume that $h_+(m_\xi)=0$, evaluating this in Equation~\eqref{eq:h+RDE} yields
    \[ 0=h_X(m_\xi)\hat{\phi}\Big(1-\mathbb{E}_{W\sim \omega}[h_-(W-m_\xi)]\Big)   .\]
    Since $h_X$ is atomic at $m_\xi$, we have that $h_X(m_\xi)>0$ and hence
    \[ 0=\hat{\phi}\Big(1-\mathbb{E}_{W\sim \omega}[h_-(W-m_\xi)]\Big)    \]
    which implies that 
    \[ 1= \mathbb{E}_{W\sim \omega}[h_-(W-m_\xi)]. \]
    So almost surely over $W \sim \omega$, we have that
    \[ 1=h_-(W-m_\xi).    \]
Using this and the right-hand side of Equation~\eqref{eq:h-RDE} with $t=W-m_\xi$, we obtain that $W-$almost surely,
    \[ 1= h_X(W-m_\xi)\hat{\phi}\Big(1-\mathbb{E}_{W' \sim \omega}[h_+(W'-W+m_\xi)] \Big) . \]
    Note that this is already a contradiction if $m_\xi>0$. With no assumption on $m_\xi$, this implies that $W-$almost surely,
    \[ 0=\mathbb{E}_{W'\sim \omega}[h_+(W'-W+m_\xi)]. \]
   This further implies that $(W,W')-$almost surely, 
    \[ 0=h_+(W'-W+m_\xi).  \]
    Inducting this argument leads to the conclusion that for any $N\in \mathbb{N}$ and any $(W_l)_{l \in \mathbb{N}_0}, (W_l')_{l \in \mathbb{N}_0}$  independent random variables with distribution $\omega$, we have
    \[ 0=h_+\left( \sum_{l=1}^N (W_l'-W_l)+m_\xi  \right) . \]
    Using that $h_+$ is non-decreasing and that $W_l,W_l'$ are continuous and hence not almost surely constant, this implies that $h_+=0$ on $\mathbb{R}$, which is clearly a contradiction.    
\end{proof}

\section*{Appendix B}\label{AppendixB}
\setcounter{section}{1}          
\renewcommand\thesection{B}
\setcounter{lemma}{0}
\setcounter{prop}{0}
\setcounter{remark}{0}
\setcounter{claimA}{0}
\renewcommand\thelemma{\thesection.\arabic{lemma}}
\renewcommand\theprop{\thesection.\arabic{proposition}}
\renewcommand\theremark{\thesection.\arabic{remark}}

This appendix's goal is twofold:
\begin{enumerate}
    \item[(O1)] In either the atomic or non-atomic case for $\xi$, we will show that the condition of $S_{\omega}$ having a non-empty interval below its maximum $M_{\omega}$ is sufficient so that $h_+(t_0)=h_-(t_0)$ at any non-trivial point $t_0$ implies the full equality $h_+=h_-$ (see Proposition~\ref{prop:objective1}).
    \item[(O2)] Since we already have that $h_+(m_{\xi})=h_-(m_{\xi})$ in the atomic case (see ~\eqref{eq:equality_at_minimum} in Case~\ref{case:atomic}), it remains to show that such a non-trivial point exists in the case where $S_{\xi}$ is non-convex and non-atomic  (see Proposition~\ref{prop:objective2}).
\end{enumerate}

We will proceed in the order announced above. Recall we assumed  that $M_{\omega}-m_{\xi} > m_{\xi}$. 
\begin{propA}\label{prop:objective1}
Assume that there exists $t_0\in [m_\xi,M_{\omega}-m_{\xi}]$ such that $0<h_+(t_0)=h_-(t_0)<1$. Then $h_+=h_-$ on $[m_\xi,M_{\omega}-m_{\xi}]$.
\end{propA}
\begin{proof}[Proof of Proposition~\ref{prop:objective1}]
Let us do the non-atomic case with $h_X(m_{\xi})=0$ and explain the minor adaptation in the atomic case at the end. 
Using the assumption $h_+(t_0)=h_-(t_0)$, we evaluate RDE~\eqref{eq:h+RDE}--\eqref{eq:h-RDE} at $t_0$ to obtain 
\begin{align}\label{eq:eval1}
h_X(t_0)\hat{\phi}\Big(1-\mathbb{E}_{W_1\sim \omega}[h_-(W_1-t_0) ]\Big)=h_X(t_0)\hat{\phi}\Big(1-\mathbb{E}_{W_1\sim \omega}[h_+(W_1-t_0) ]\Big).     
\end{align} 
Since $t_0> m_{\xi}$ as $h_+(t_0)>0$, we must have $h_X(t_0)>0$ from Equation~\eqref{eq:h+RDE}. 
Since $\hat{\phi}$ is increasing, it is invertible, so from Equation~\eqref{eq:eval1} we have
\[ \mathbb{E}_{W_1\sim \omega}[h_-(W_1-t_0) ]=\mathbb{E}_{W_1\sim \omega}[h_+(W_1-t_0) ]. \]
Owing to the fact that $h_+ \leq h_-$,  this implies that $W_1-$almost surely,
\[ h_-(W_1-t_0)=h_+(W_1-t_0).\]
Using this and RDE~\eqref{eq:h+RDE}--\eqref{eq:h-RDE}, we now iterate again to obtain
\begin{align}
&h_X(W_1-t_0)\hat{\phi}\Big(1-\mathbb{E}_{W_2\sim \omega}[h_+(W_2-W_1+t_0)]\Big)\label{eq:eval2-1}\\
&= h_X(W_1-t_0)\hat{\phi}\Big(1-\mathbb{E}_{W_2\sim \omega}[h_-(W_2-W_1+t_0)]\Big). \label{eq:eval2-2} 
\end{align}
Now, either $W_1-t_0\leq m_{\xi} $, or $W_1-t_0> m_{\xi} $ and so $h_X(W_1-t_0)>0$. Thus, conditionally on $W_1-t_0 > m_{\xi}$, we can once again cancel $h_X(W_1-t_0)$ out from \eqref{eq:eval2-1}--\eqref{eq:eval2-2}, and using the invertibility of $\hat{\phi}$ we obtain
\[ \mathbb{E}_{W_2\sim \omega}[h_+(W_2-W_1+t_0)]=\mathbb{E}_{W_2\sim \omega}[h_-(W_2-W_1+t_0)].\]
Then, remembering that $h_+ \leq h_-$, we have that $(W_2,W_1)-$almost surely
\[\text{ either } \quad W_1-t_0 \leq m_{\xi} \quad \text{ or } \quad   h_+(W_2-W_1+t_0)=h_-(W_2-W_1+t_0). \]

Let us repeat once more to see how the $h_X>0$ condition evolves, we have that, conditionally on $W_1-t_0 > m_{\xi}$,
\begin{align*}\label{eq:eval3}
&h_X(W_2-W_1+t_0)\hat{\phi}\Big(1-\mathbb{E}_{W_3\sim \omega}[h_-(W_3-W_2+W_1-t_0) ]\Big)\\
&=h_X(W_2-W_1+t_0)\hat{\phi}\Big(1-\mathbb{E}_{W_3\sim \omega}[h_+(W_3-W_2+W_1-t_0) ]\Big).  
\end{align*}
Next, conditionally on $W_2-W_1+t_0>m_\xi$, we still get the next step, that is, $W_3-$almost surely,
\[ h_-(W_3-W_2+W_1-t_0)=h_+(W_3-W_2+W_1-t_0). \]
So we get that $(W_3,W_2,W_1)-$almost surely, 
\[\text{ either } \quad  W_2-W_1+t_0  \leq m_{\xi} \quad \text{ or } \quad  h_-(W_3-W_2+W_1-t_0)=h_+(W_3-W_2+W_1-t_0).  \]

We are now ready to repeat the argument, take $(W_l)_{l\in \mathbb{N}_0}$ an i.i.d. sequence of law $\omega$. For any $N  \in \mathbb{N}$  define the random variable 
\[ A_N:=\sum_{l=0}^{N-1}(-1)^l W_{N-l} + (-1)^{N}t_0  \] and the associated stopping time
\[ N_0:=\inf \{ N'\in \mathbb{N}:  A_{N'} \leq m_{\xi} \} .\]
Prior to $N_0$, we can repeat the above process without $h_X(A_N)$ evaluating to zero and thus propagate the equality between $h_+(A_N)$ and $h_-(A_N)$. We thus get that $(W_l)_{l \in \mathbb{N}_0}-$almost surely, for all $1\leq N' \leq N_0$,
\[ h_+(A_{N'})=h_-(A_{N'}).\]

We now claim that the random variables $(A_{N'})_{1\leq N' \leq N_0}$ will span a dense set in $[m_\xi,M_{\omega}-m_{\xi}]$.

    \begin{claimA}\label{claim:supportalternate}
        Let $x_0$ be a point in $S_{\omega}$ with $x_0>t_0$ (where $t_0$ is as in Proposition~\ref{prop:objective1}) and assume that $S_{\omega}$ contains an open interval of the form $(x_0-\eps,x_0)$. Then over $N\geq 2$,  conditionally on that for any $1\leq N'<N$, 
        \[ A_{N'}= \sum_{l=0}^{N'-1} (-1)^l W_{N'-l} + (-1)^{N'}t_0 >  m_{\xi},  \] 
        the random variables
        \[ A_N= \sum_{l=0}^{N-1} (-1)^l W_{N-l} +(-1)^{N} t_0 \]       
        spans $[m_{\xi}, x_0-m_{\xi}]$.
    \end{claimA}
    \begin{proof}[Proof of Claim~\ref{claim:supportalternate}]
        Let $\eps>0$ be as stated in the hypothesis, such that $(x_0-\eps,x_0] \subseteq S_{\omega}$.
        For $N \in \mathbb{N}$, let $\mathrm{Span}_N$ be the support of $A_N$
        conditionally on that $N_0 > N$. We also set $\mathrm{Span}_0:=\{t_0\}$ which is consistent.

        The recursion satisfied by $\mathrm{Span}_N$ is then
        \[ \mathrm{Span}_{N+1} = \bigg(\bigg\{ w-s: w \in S_{\omega}, s \in \mathrm{Span}_{N} \bigg\} \cap [m_{\xi},+\infty] \bigg) \cup \mathrm{Span}_{N}  . \]

        The sequence $(\mathrm{Span}_N)_{N \in \mathbb{N}_0}$ thus grows until it eventually reaches
        \[ \mathrm{Span}_{\infty} := \bigcup_{N \in \mathbb{N}_0} \mathrm{Span}_{N} . \]
        Let us compute the first couple terms of $\mathrm{Span}_N$ in the simple case where we forget all the mass in $S_{\omega}$ outside of $[x_0-\eps,x_0]$.
        \begin{align*}
            \mathrm{Span}_0&=\{t_0\}, \\
            \mathrm{Span}_1&= \{ w-t_0: x_0-\eps \leq w \leq x_0, w-t_0 \geq m_{\xi}\} \cup \{t_0\} \\
             &= [\max(m_{\xi},x_0-t_0-\eps), x_0-t_0] \cup \{t_0\}, \\
             \mathrm{Span}_2&=\{w-s: x_0-\eps \leq w \leq x_0, \max(m_{\xi},x_0-t_0-\eps)\leq s \leq x_0-t_0, w-s\geq m_{\xi} \} \cup \mathrm{Span}_1 \\
             &=[\max(m_{\xi},x_0-\eps-(x_0-t_0)), x_0-\max(m_{\xi},x_0-t_0-\eps)] \cup \mathrm{Span}_1 \\
             &=[\max(m_{\xi},t_0-\eps),\min(x_0-m_{\xi},t_0+\eps)] \cup \mathrm{Span}_1.
        \end{align*}
    The behaviour then becomes clear, the set $\mathrm{Span}_2$ grew to the right and left of $t_0$ by $\eps$ under the action of the interval of length $\eps$ in $S_{\omega}$, and this will continue in the even sequence $\mathrm{Span}_{2N}$ until either $m_{\xi}$ or $x_0-m_{\xi}$ is reached. The odd sequence $\mathrm{Span}_{2N+1}$ will grow similarly around $x_0-t_0$ instead, so that they will simultaneously cover two intervals of the form $[m_{\xi},m_{\xi}+c] \cup [x_0-m_{\xi}-c,x_0-{m_{\xi}}]$ after at most $\lfloor\frac{2}{\eps}\max(x_0-t_0,t_0-m_{\xi})\rfloor$ amount of steps for some constant $c>0$. Now we just need to show that being limited on the left by $m_{\xi}$ and on the right by $x_0-m_{\xi}$ does not stop the sequence from growing and covering the middle gap had we started with $t_0$ outside of the middle third of $[m_{\xi},x_0-m_{\xi}]$. Let us thus assume that $\mathrm{Span}_{N_0}=[m_{\xi},m_{\xi}+c] \cup [x_0-m_{\xi}-c,x_0-m_{m_{\xi}}]$ for some $N_0>0$ and compute the subsequent steps:
    \begin{align*}
        \mathrm{Span}_{N_0}=& [m_{\xi},m_{\xi}+c] \cup [x_0-m_{\xi}-c,x_0-m_{{\xi}}] ,\\
        \mathrm{Span}_{N_0+1}=&\bigg\{w-s: x_0-\eps \leq w \leq x_0, m_{\xi}\leq s \leq m_{\xi}+ c, w-s \geq m_{\xi} \bigg\} \\
        \cup& \bigg\{w-s: x_0-\eps \leq w \leq x_0, x_0-m_{\xi}-c\leq s \leq x_0-m_{\xi}, w-s \geq m_{\xi} \bigg\}
        \cup \mathrm{Span}_{N_0} \\
        = &[\max(m_{\xi},x_0-\eps-m_{\xi}-c),x_0-m_{\xi}] \\\cup &[\max(m_{\xi},x_0-\eps-x_0+m_{\xi}),x_0-x_0+m_{\xi}+c]
        \cup  \mathrm{Span}_{N_0} \\
        =&[\max(m_{\xi}, x_0-m_{\xi}-c-\eps),x_0-m_{\xi}] \cup [m_{\xi},m_{\xi}+c] \cup \mathrm{Span}_{N_0}, 
         \end{align*}
         and 
          \begin{align*}
        \mathrm{Span}_{N_0+2}=&[\max(m_{\xi}, x_0-m_{\xi}-c-\eps),x_0-m_{\xi}] \\
        \cup &[\max(m_{\xi},x_0-\eps-(x_0-m_{\xi}), x_0-\max(m_{\xi},x_0-m_{\xi}-c-\eps) ] \cup \mathrm{Span}_{N_0+1}\\
        =&[\max(m_{\xi}, x_0-m_{\xi}-c-\eps),x_0-m_{\xi}] \cup [m_{\xi}, \min(x_0-m_{\xi},m_{\xi}+c+\eps) ] \cup \mathrm{Span}_{N_0+1}.        
    \end{align*}
    So we see that $\mathrm{Span}_{N_0+2}$ grows by $\eps$ to the right of $m_{\xi}+c$ and to the left of $x_0-m_{\xi}-c$ compared to $\mathrm{Span}_{N_0+2}$, so it will eventually cover the whole interval $[m_{\xi},x_0-m_{\xi}]$.
    Naturally, if $S_{\omega}$ contains $[x_0,x_0-\eps]$ instead of being exactly equal, then the sequence $(\mathrm{Span}_N)_{N \geq 0}$ can only grow faster than what we just computed, so the conclusion holds.
    \end{proof}

    We continue with the proof of Proposition~\ref{prop:objective1}.    Since we assumed that $S_{\omega}$ contains an open interval of the form $(M_{\omega}-\eps,M_{\omega})$, Claim~\ref{claim:supportalternate} implies that $h_+=h_-$ on some set that is dense in $[m_{\xi},M_{\omega}-m_{\xi}]$. Using the fact that $h_+,h_-$ are non-decreasing (and hence have at most countable discontinuities) and right-continuous, we deduce that $$h_+=h_-\quad \text{on} \quad  [m_{\xi},M_{\omega}-m_{\xi}],$$
    as desired.
    
    Finally, if we assumed $\xi$ to be atomic at $m_{\xi}$, the same proof holds except that the stopping condition becomes $A_{N_0}< m_{\xi}$ instead of $A_{N_0} \leq m_{\xi}$, which changes nothing with regards to the support as we are interested in the closure. This completes the proof of Proposition~\ref{prop:objective1}.
\end{proof} 

We are now done with our first objective (O1). We now proceed with the second one (O2) in the non-atomic case. We thus assume that $h_X(m_{\xi})=0$ in the remainder of Appendix~B.

\begin{propA}\label{prop:objective2}
There exists $t_0 \in (m_{\xi}, M_{\omega}-m_{\xi})$ such $0 < h_+(t_0) = h_-(t_0) < 1$.
\end{propA}

\begin{proof}[Proof of Proposition~\ref{prop:objective2}]
The idea behind this proposition is that we would like to mimic the simple case in Section~\ref{sec:mainprop} where one could just invert $h_X$ in the identity 
\[ h_X \circ h_+^{-1} = h_X \circ h_-^{-1} \quad \text{a.e.} \quad \text{on}\quad  [0,1]\] to obtain the desired conclusion. 
There are roughly two hurdles:
\begin{enumerate}
    \item The identity only holds a.e. on $[0,1]$, it could \textit{a priori} be possible that $h_+^{-1}$ and $h_-^{-1}$ map almost all of $[0,1]$ to a set where $h_X$ is flat (on its connected components), so that all the actual information is contained within a set of measure zero.
    \item The cumulative distribution function $h_X$ is not invertible in general. Since we only assume that $\xi$ is continuous, it could be highly pathological (e.g., Cantor staircase or similar functions).
\end{enumerate}
To circumvent such obstacles, we utilise the following ideas: 
\begin{enumerate}
\item It is true that $\xi$ can be very pathological, but the little convex interval under $M_{\omega}$ in $S_{\omega}$ is sufficient to deduce that $S_{\zeta_+}$ and $S_{\zeta_-}$ will cover the entire interval $[m_{\xi},M_{\omega}-m_{\xi}]$, and thus, neither $h_+$ nor $h_-$ have flat pieces in that interval. This implies that $h_+^{-1}$ and $h_-^{-1}$ are continuous functions on the range of $h_+$ and $h_-$ on $[m_{\xi},M_{\omega}-m_{\xi}]$. This allows us to deduce that $h_X \circ h_+^{-1}= h_X \circ h_-^{-1}$ holds everywhere---not just almost everywhere---by density and right-continuity. 
\item The topological interior of the support of $S_{\xi}$ can be empty, but the subset of $S_{\xi}$ where $h_X$ is not invertible is at most countable, even if $S_{\xi}$ itself can be of measure zero. So we can pretend that $h_X$ is invertible by removing at most countable points from $S_{\xi}$. Since $S_{\xi}$ must be uncountable by virtue of having no atoms, the remaining set must still be uncountable and hence nonempty. This set can be of measure zero and not dense in $[m_{\xi},M_{\omega}-m_{\xi}]$, but we showed previously in Proposition~\ref{prop:objective1} that a single non-trivial point suffices.
\end{enumerate}
First we determine the supports $S_{\zeta_+}, S_{\zeta_-}$ of $\zeta_+, \zeta_-$.
\begin{claimA}\label{claim:supportdecomposition}
    We have that $S_{\zeta_+}=S_{\zeta_-}=[m_{\xi},M_{\omega}-m_{\xi}] \cup S_{\xi}$.
\end{claimA}
\begin{proof}[Proof of Claim~\ref{claim:supportdecomposition}]
   Since we assume that $M_{\omega}-m_{\xi} \geq m_{\xi}$,  the support of $\zeta_+,\zeta_-$ cannot be contained in $[M_{\omega}-m_{\xi},+\infty)$ as this support is unstable under iterations of RDE~\eqref{eq:h+RDE}--\eqref{eq:h-RDE}. So there exist at least two points $t_{0,+}, t_{0,-} \in [m_{\xi},M_{\omega}-m_{\xi})$ such that $t_{0,+} \in S_{\zeta_+}$ and  $t_{0,-} \in S_{\zeta_-}$. Denote by $o(\pm)$ in an expression to mean that $\pm=+$ or $\pm=-$ and $o(\pm)$ is the opposite of $\pm$. Define $\mathrm{Span}_{N,+}$ and $\mathrm{Span}_{N,-}$ with:
    \begin{align}
        \mathrm{Span}_{0,+}&=\{t_{0,+}\} ,\\
        \mathrm{Span}_{0,-}&=\{t_{0,-}\} ,\\
        \mathrm{Span}_{N+1,\pm}&= \bigg(\bigg\{ w-s: w \in S_{\omega}, s \in \mathrm{Span}_{N,o(\pm)} \bigg\} \cap [m_{\xi},+\infty] \bigg) \cup \mathrm{Span}_{N,o(\pm)}. 
    \end{align}
    Then $\mathrm{Span}_{N,\pm} \subseteq S_{\zeta_{\pm}}$ for all $N\in \mathbb{N}_0$ by using the fact that $\xi$ must attribute positive mass to arbitrary right-neighbourhoods of $m_{\xi}$ (regardless of the atomic or non-atomic case for $\xi$ at $m_{\xi}$) and iterating RDE~\eqref{eq:h+RDE}-~\eqref{eq:h-RDE}. Claim~\ref{claim:supportalternate} thus implies that $\bigcup_{N \in \mathbb{N}_0} \mathrm{Span}_{N,\pm}=[m_{\xi},M_{\omega}-m_{\xi}]$.
    Finally, RDE~\eqref{eq:h+RDE}--~\eqref{eq:h-RDE} also imply that $h_{\pm}(t)=h_X(t)$ whenever $t > M_{\omega}-m_{\xi}$ as $\xi$ will always dominate in the maximum, so 
    \begin{align}
 h_+=h_-=h_X   \quad \text{on} \quad      S_{\xi}\cap [M_{\omega}-m_{\xi},+\infty)=S_{\zeta_{\pm}} \cap [M_{\omega}-m_{\xi},+\infty).\label{eq:hpm=hX}
    \end{align}

\end{proof}

Consider the following decomposition of $S_{\xi}$ by setting
\begin{align}
    S_{\xi,\mathrm{inner}}&:=\left\{ x \in S_{\xi}: \forall \eps >0, h_X(x-\eps)>x>h_X(x-\eps) \right\}, \\
    S_{\xi, \mathrm{bound}}&:=\left\{ x \in S_{\xi}: \exists  \eps_0 >0, h_X(x-\eps_0)=x \text{ or } h_X(x+\eps_0)=x \right\}.
\end{align}
The set $S_{\xi,\mathrm{inner}}$ is not the topological interior of $S_{\xi}$, but it is the reasonable maximal subset of $S_{\xi}$ where $h_X$ is invertible by removing the endpoint of every flat piece. Furthermore, $S_{\xi,\mathrm{bound}}$ is at most countable by mapping every $x \in S_{\xi,\mathrm{bound}}$ to an arbitrarily picked rational number in $[x,x-\frac{\eps_0}{3}]$ or $[x,x+\frac{\eps_0}{2}]$ depending on which condition is satisfied.
    
It will be useful to prove the next fact to conclude the proof of Proposition~\ref{prop:objective2}:
\begin{claimA}\label{claim:existence_of_point}
     If $h_X \circ h_{+}^{-1}=h_X \circ h_-^{-1}$ almost everywhere, then there exists $t_0\in [m_{\xi},M_{\omega}-m_{\xi}]$ such that $0<h_+(t_0)=h_-(t_0)<1$.
\end{claimA}
\begin{proof}[Proof of Claim~\ref{claim:existence_of_point}]
   We know that outside of some set $E$ of measure zero,
   \[h_X \circ h_+^{-1}=h_X \circ h_-^{-1}.\]
   Now thanks to Claim~\ref{claim:supportdecomposition}, we know that $h_+,h_-$ are strictly increasing on $[m_{\xi},M_{\omega}-m_{\xi}]$. Since they are continuous (and increasing), their range over this interval must also be an interval of some form $[0,a_+]$ and $[0,a_-]$ with $a_+ \leq a_-$.
   In particular, this implies that $h_{+}^{-1}$ and $h_{-}^{-1}$ are continuous increasing maps of $[0,a_+]$ into $[m_{\xi},M_{\omega}-m_{\xi}]$.
    Specifically, on $[0,a_+]$, the functions $h_X\circ h_+^{-1}$ and $h_X\circ h_-^{-1}$ are increasing and right-continuous. For such functions, equality almost everywhere implies equality by a density argument, hence $h_X\circ h_+^{-1}=h_X\circ h_-^{-1}$ on $[0,a_+]$.

   Next, since we are in the case where $\xi$ is non-atomic, $m_{\xi}$ cannot be an isolated point of $S_{\xi}$. Moreover, any neighbourhood of the form $[m_{\xi},m_{\xi}+\varepsilon]$ for $\eps>0$ must intersect $S_{\xi}$ uncountable many times, or else $\xi$ would attribute zero mass to $[m_{\xi},m_{\xi}+\varepsilon]$ by union bound, which contradicts $m_{\xi}$ being in $S_{\xi}$. Since we showed before this claim that $S_{\xi,\mathrm{bound}}$ is at most countable, we deduce that $[m_{\xi},m_{\xi}+\eps]\cap S_{\xi,\mathrm{inner}}$ is uncountable thus non empty for every $\eps>0$.
   Finally, fixing $\eps=M_{\omega}-2m_{\xi}$ leads to
   \[ h_+\left([m_{\xi},M_{\omega}-m_{\xi}] \cap S_{\xi,\mathrm{inner}}\right)= [0,a_+] \cap h_+(S_{\xi,\mathrm{inner}}) \neq \emptyset. \]
   On this set, we thus both have that:
   \begin{enumerate}
       \item The equality \[h_X\circ h_+^{-1}=h_X\circ h_-^{-1}\] holds everywhere and not just almost surely.
       \item $h_X$ is invertible for all values of $h_+^{-1}$ by virtue of $h_+^{-1}$ evaluating into $S_{\xi,\mathrm{inner}}$.
   \end{enumerate}
   We can thus invert $h_X$ to get that $h_+^{-1}=h_-^{-1}$ on $[0,a_+] \cap h_+(S_{\xi,\mathrm{inner}})$. Since both $h_+^{-1}$ and $h_-^{-1}$ are continuous increasing on $[0,a_+]$, we can invert to obtain $h_+=h_-$ on $h_+^{-1}([0,a_+]\cap h_+(S_{\xi,\mathrm{inner}}))=[m_{\xi},M_{\omega}-m_{\xi}] \cap S_{\xi,\mathrm{inner}}$, which contains at least one point $t_0$ bigger than $m_{\xi}$ and smaller than $M_{\omega}-m_{\xi}$, hence $0<h_+(t_0)<1$.
   \end{proof}

Recall from \eqref{eq:h_x} that we have 
\begin{equation*}  
h_X \circ h_+^{-1}= h_X \circ h_-^{-1}  \quad \text{a.e.} \quad \text{ on } \quad [0,1],
\end{equation*}
which together with Claim~\ref{claim:existence_of_point} completes the proof of Proposition~\ref{prop:objective2}.
\end{proof}

\smallskip
To conclude, we already know that $h_+=h_-=h_X$ on $(M_{\omega}-m_{\xi},+\infty)$ regardless of $\xi$ having an atom at $m_{\xi}$ or not. In the atomic case where $h_X(m_{\xi})=\alpha>0$, Proposition~\ref{prop:objective1} combined with identity~\eqref{eq:equality_at_minimum} gives that $h_+=h_-$ on $[m_{\xi},M_{\omega}-m_{\xi}]$, completing the equality. In the case where $h_X(m_{\xi})=0$, we combine Propositions~\ref{prop:objective1} and~\ref{prop:objective2}
instead. We have thus obtained that $h_+=h_-$ on $[m_{\xi},+\infty)$ in both cases, completing the proof of Proposition~\ref{prop:mainprop}.

\smallskip
    \begin{remarkA}
     The condition on $S_{\omega}$ was mostly chosen to avoid the pathological case where $S_{\omega}$ and $S_{\xi}$ are such that the sets $(\mathrm{Span}_{N,\zeta})_{N \geq 0}$ defined recursively as before but starting with 
     \begin{align*}
         \mathrm{Span}_{0,\zeta}=\{m_{\xi}\}
     \end{align*}
     end up not covering the entire interval $[m_{\xi}, M_{\omega}-m_{\xi}]$, which could potentially happen if they form some a strict subset $K$ of  $[m_{\xi},M_{\omega}-m_{\xi}]$ verifying
     \[ K= \left(S_{\xi} \cup \{w-s, w \in S_{\omega}, s \in K \}\right) \cap [m_{\xi}, M_{\omega}-m_{\xi}]  .\]
     In this case, the restriction of $\zeta_+$ and $\zeta_-$ to $[m_{\xi},M_{\omega}-m_{\xi}]$ would be supported on $K$. The authors do believe the conclusion may still hold in this case, but will most likely require many more technicalities to deal with the fact that $h_+$ and $h_-$ will be pathological instead of nicely strictly increasing in a right-neighbourhood of $m_{\xi}$.
    \end{remarkA}

\end{document}